\documentclass[11pt,a4paper]{amsart}
\usepackage{amscd,amssymb,amsmath,latexsym}
\newcommand{\Sres}{{\operatorname{Sres}}}
\newcommand{\Res}{{\operatorname{Res}}}
\newcommand{\Syl}{{\operatorname{Syl}}}
\newcommand{\SylM}{{\operatorname{SylM}}}

\newcommand{\coeff}{{\operatorname{coeff}}}
\newcommand{\sg}{{\operatorname{sg}}}

\theoremstyle{plain}

\newtheorem{theorem}{Theorem}[section]
\newtheorem{lemma}[theorem]{Lemma}
\newtheorem{proposition}[theorem]{Proposition}

\newtheorem{definition}[theorem]{Definition}

\newtheorem{example}[theorem]{Example}

\def\N{{\mathbb N}}
\def\Z{{\mathbb Z}}
\def\K{{\mathbb K}}

           {\vspace{3.3mm}
           \noindent{\bf #1}\it}%
           {\vspace{3.3mm}}

\newenvironment{ac}{\noindent{\bf Acknowledgements:}}{}

\begin{document}

\title{Closed Formula for Univariate Subresultants in Multiple Roots}

\author[C. D'Andrea]{Carlos D'Andrea}
\address{Department de Matem\`atiques i Inform\`atica, Facultat de
Matem\`atiques i Inform\`atica, Universitat de Barcelona, Gran Via de les Corts
Catalanes, 585; 08007 Spain.}
\email{cdandrea@ub.edu}
\urladdr{http://www.ub.edu/arcades/cdandrea.html}

\author[T. Krick]{Teresa Krick}
\address{Departamento de Matem\'atica, Facultad de
Ciencias Exactas y Naturales  and IMAS,
CONICET, Universidad de Buenos Aires,  Argentina.}
\email{krick@dm.uba.ar}
\urladdr{http://mate.dm.uba.ar/\~\,krick}

\author[A. Szanto]{Agnes Szanto}
\address{Department of Mathematics, North Carolina State
University, Raleigh, NC 27695 USA.}
\email{aszanto@ncsu.edu}
\urladdr{www4.ncsu.edu/\~\,aszanto}

\author[M. Valdettaro]{Marcelo Valdettaro}
\address{Departamento de Matem\'atica, Facultad de
Ciencias Exactas y Naturales, Universidad de Buenos Aires, Argentina.}
\email{mvaldett@dm.uba.ar}

\begin{abstract}
We generalize Sylvester single sums  to multisets and show that these sums compute  subresultants of two univariate polynomials as a function of their roots independently of their multiplicity structure.  This is the first closed formula for subresultants in terms of roots that works for arbitrary polynomials, previous efforts only handled special cases. Our extension involves in some cases confluent Schur polynomials and is obtained by using  multivariate symmetric interpolation via an Exchange Lemma. \end{abstract}


\maketitle

\section{Introduction}
 Let $\K$ be a field. Given two finite sets $A,B\subset \K$ of cardinalities $m$ and $n$ respectively,  and  $0\leq d\leq m$, J.J. Sylvester introduced in \cite{sylv40b} the following  {\em single sum}:
 \begin{equation}\label{defsinglesum}
 \Syl_{d}(A,B)(x)=\sum_{{A'\subset A,
|A'|=d}} \frac{\mathcal{R}(A\backslash A',B)\mathcal{R}(x,A')}{\mathcal{R}(A',A\backslash A')},
\end{equation}
where $\mathcal{R}(X,Y):=\prod_{\substack{x\in X\\
y\in Y}}(x-y)$, with the convention that $\mathcal{R}(X,Y)=1$ if $X=\emptyset$ or $Y=\emptyset$, and $A\backslash A'$ denotes as usual the set difference.

For  $f:=f_m x^m+\cdots + f_0,\,g:=g_n x^n+\cdots + g_0\in\K[x],$  and $0\le d\le  \min\{m,n\}$ when $m\ne n$ or $0\le d<m=n,$ Sylvester also introduced in \cite{sylv39,sylv40} the {\em $d$-th order subresultant} $Sres_d(f,g)(x)\in \K[x]:$  \begin{equation*}\label{srs}
\Sres_d(f,g)(x):= \det \begin{array}{|cccccc|c}
\multicolumn{6}{c}{\scriptstyle{m+n-2d}}&\\
\cline{1-6}
f_{m} & \cdots & &\cdots &f_{d+1-(n-d-1)} &x^{n-d-1}f& \\
&  \ddots & &&\vdots  & \vdots& \scriptstyle{n-d}\\
& & f_m& \dots &f_{d+1}&f& \\
\cline{1-6}
g_{n} &\cdots & &\cdots  &g_{d+1-(m-d-1)}  &x^{m-d-1}g&\\
&\ddots && &\vdots  &\vdots  &\scriptstyle{m-d}\\
&&g_{n} & \cdots &  g_{d+1} &g&\\
\cline{1-6} \multicolumn{2}{c}{}
\end{array}.
\end{equation*}

When
\begin{align*}
f&=\prod_{a\in A}(x-a) \quad \mbox{and} \quad g=\prod_{b\in B}(x-b),\end{align*}
that is, all roots of $f$ and $g$ are simple roots,
the following relation between Sylvester single sums and subresultants was stated  in \cite{sylv40b} and then established in  \cite[Section 2]{Syl}: for $d\le \min\{m,n\}$ when $m\ne n,$ or $d< m=n$,
 \begin{equation}\label{sumasimple}\Sres_d(f,g)(x)= (-1)^{d(m-d)} \Syl_{d}(A,B)(x).
 \end{equation}
The interest for having formulas ``in roots'' for subresultants comes from the fact that it is well-known that  for $0\leq d\leq \min\{m,n\},\, \Sres_d(f,g)(x)$ is ``the'' polynomial of degree $d$ appearing in the polynomial remainder sequence starting with $f$ and $g,$ and whose last element is $\gcd(f,g).$ In the particular case $g=f',$ the derivative of $f$, the study of the variation of signs of the elements of this sequence in a given interval leads to explicit criteria for computing the number of real roots of $f$ in that interval. This is the so-called \emph{Sturm's method} which was actually the main focus of Sylvester's papers in \cite{sylv39,Syl}. These results allowed Sylvester to obtain formulas for Sturm's auxiliary functions in terms of the roots  of $f$, and these expressions  became well-known in his lifetime (see \cite[Art. 35]{Syl}). 

Note that \eqref{sumasimple} can be considered as a ``Poisson formula'' for the subresultant, generalizing the well known Poisson formula for resultants $$\Res(f,g)=\Sres_0(f,g) = \prod_{a\in A} g(a),$$ as it describes it in terms of the values of $g$ in the roots of $f$:
$$
\Sres_d(f,g)(x)= (-1)^{d(m-d)}\sum_{A'\subset A, |A'|=d}
\frac{\prod_{a\in A\backslash A'}g(a)\mathcal{R}(x,A')}{\mathcal{R}(A',A\backslash A')},
$$
but this equality only holds in the case where the roots of $f$ are all simple, i.e. when $f$ is squarefree, unlike the classical Poisson formula for resultants.

Even though there is a long history in the study of the connection between subresultants and Sylvester sums in the simple root case (see for instance \cite{Bor60,ApJo06,Cha90,Hon99,LaPr01,DTGV04,DHKS07, RoSz11,KS01,KSV16} and the references therein), little is known about extensions when the roots of $f$ and $g$ have multiplicities. As noted above, the generalization is not straightforward, since some denominators in the Sylvester sums turn zero in case of root multiplicities, despite the fact that the left hand side of \eqref{sumasimple} is well defined even in these cases. The article \cite{DKS2013} describes formulas in terms of the roots of arbitrary polynomials but only for the  order $d=1$ and $d=\min\{m, n\}-1$ subresultants, while  formulas for the subresultants of any order $d$ but only for  the extremal case when $f=(x-a)^m$ and $g=(x-b)^n$ in terms of $a$ and $b$ have been developed in \cite{BDKSV17}.

The present paper is the first one  to give expressions for  subresultants of arbitrary polynomials $f$ and $g$ and arbitrary values of $0\leq d\leq \min\{m,n\}$ that are a generalization of the classical Sylvester single sums.

To this aim,  in Definitions \ref{def-bigd} and \ref{SylM}, we present  $\SylM_{d}(A,B)(x)$, a generalization of the notion of $\Syl_{d}(A,B)(x)$ for multisets (sets where repeated elements are allowed). First, in Definition \ref{def-bigd}, we consider the case when   $A$ and $B$ are multisets and  $d$  sufficiently large. Then, in Definition \ref{SylM}, we extend our definition to any $d$ with the aid of Schur functions. Both of these definitions coincides with $\Syl_{d}(A,B)(x)$ defined in (\ref{defsinglesum}) when $A$ is a set, and  in Theorems  \ref{sub-bigd-2}  and \ref{alld2} we show  that our definitions also satisfy -as desired- Identity~\eqref{sumasimple}.

In order to state our main results, we first introduce some notation we extend from sets to multisets.
Given a multiset $X\subset \K$ , we denote with $|X|$ its length (the number of elements counted with multiplicities). If $X'\subset X$ are  multisets, then  $X\setminus X'$ is the multiset difference, defined by the elements of $X$ with  multiplicities  equal to the difference between their values in $X$ and in $X'$.

\begin{definition}\label{def-bigd} Let  $A, \, B\subset\K$  be multisets with $|A|=m,\, |B|=n$ and let $\overline A\subset A$ and $\overline B\subset B$ be fixed subsets of the sets of distinct elements in $A$ and $B$ respectively, and set  $m':=m-|\overline A|$ and $n'=n-|\overline B|$.
 For any $d$ such that $m'+n'\le d \le \min\{m,n\}$ if $m\ne n$ or $m'+n'\le d<m=n$, we define
{\small   {\begin{align*} &\SylM_{d}(A,B)(x):=\\
  &
   (-1)^{m'(m-d)}\sum_{\substack{A'\subset \overline A\\|A'|=d-m' }}
\sum_{\substack{B'\subset \overline B\\|B'|=m' }} \frac{\mathcal{R}(A\backslash \overline A, \overline B\backslash B')\mathcal{R}(\overline A\backslash A',  B\backslash B')\mathcal{R}(x,A')\mathcal{R}(x,B')}{ \mathcal{R}(A', \overline A\backslash A')\mathcal{R}(  B', \overline B\backslash B')
}.\end{align*}}}
\end{definition}
It is straightforward to verify that  when $A$ is a set and $\overline{A}:=A$, i.e.    $m'=0$, then  $\SylM_{d}(A,B)(x)$ boils down to  $\Syl_{d}(A,B)(x)$, the usual Sylvester  sum which appears in \eqref{defsinglesum}. We also note here that the definition of $\SylM_{d}(A,B)(x)$ depends on the choice of $\overline A\subset A $ and $\overline B\subset B$, but since ultimately we prove that they all agree with the subresultant independently of the choice of $\overline A $ and $\overline B$, we  do not indicate this dependence in the notation for the sake of simplicity.

We then have:
\begin{theorem}\label{sub-bigd-2}
Let $f,g\in \K[x]$  be monic polynomials of degrees $m$ and $n$, with multisets of roots $A$ and $B$. Let    $\overline A$ and $\overline B$ be subset of the sets of distinct roots of $f$ and $g$, respectively, and set $m':=m-|\overline A|$ and $ n':=n-|\overline B|$.
For any $d$ such that $m'+n'\leq d\leq\min\{m,n\}$ if $m\ne n$ or $m'+n'\leq d<m=n$, we have
  {\begin{align*} &\Sres_d(f,g)(x)= (-1)^{d(m-d)} \SylM_{d}(A,B)(x).\end{align*}}
\end{theorem}

One can wonder whether the lower bound stated for  $d$ in Theorem~\ref{sub-bigd-2} is sharp, since the definition of $\SylM_{d}(A,B)(x)$ makes sense for more values of $d$, more precisely for those $d$ such that $m'\le \min\{d,|\overline B|\}$.
The next example illustrates that the result holds for $d$ in the right range and shows that the constraint  on it is necessary.

\begin{example}\label{example1}
Take $f=(x-a_1)(x-a_2)^2$ and $g=(x-b_1)^2,$ so $A=\{a_1,a_2,a_2\}$, $B=\{b_1,b_1\}$,  and we take $\overline A=\{a_1,a_2\}$ and $\overline B=\{b_1\}$. For $d=2$,  one has
 $\Sres_2(f,g)(x)= g(x)$ while   $\SylM_{2}(A,B)(x)$ equals
\begin{align*}-&\Big( \frac{(a_2-b_1)(x-a_1)(x-b_1)}{a_1-a_2}+  \frac{(a_1-b_1)(x-a_2)(x-b_1)}{a_2-a_1}\Big)\\
&\quad =\frac{\big((a_2-b_1)(x-a_1)-(a_1-b_1)(x-a_2)\big)(x-b_1)}{a_2-a_1}\\
&\quad = (x-b_1)(x-b_1)\\ & \quad =g(x),\end{align*}
so Theorem~\ref{sub-bigd-2} holds in this case, and we note that  $d=2$ is in the  range of Theorem~\ref{sub-bigd-2} since  $(3-2)+(2-1)\le 2\le \min\{3,2\}$.

Now take $f=(x-a_1)(x-a_2)^2$ and $g=(x-b_1)^3$. In this case, $A=\{a_1,a_2,a_2\}$ and $B=\{b_1,b_1,b_1\}$, and we again take  $\overline A=\{a_1,a_2\}$  and  $\overline B=\{b_1\}$. For  $d=2$ we have $\Sres_2(f,g)(x)= g(x)-f(x)$ and  $\SylM_{2}(A,B)(x)$ can still be defined according to Definition \ref{def-bigd} since $m'=1\le \min\{d,|\overline B|\}$, but it is a multiple of $x-b_1$, so the two expressions obviously do not coincide. We note that here $d=2$ is not in the range of Theorem~\ref{sub-bigd-2} since $2< (3-2)+(3-1)$.

\end{example}
To extend the definition of $\SylM_{d}(A,B)(x)$  for any  $d,$ we need to introduce confluent Schur polynomials $S_k^{(R)}(X)$, which are defined in \eqref{schg} below,   for a multiset $X$ of length $r\leq k$, by removing a subset $R$ of $k-r$ rows in the confluent Vandermonde matrix of $X$ of size $k\times r$.

\begin{definition}\label{SylM}
Let  $A, \, B\subset\K$  be multisets with $|A|=m,\, |B|=n$ and let $\overline A\subset A$ and $\overline B\subset B$ be subsets of the sets of distinct elements in $A$ and $B$ respectively, with $|\overline A|=\overline m,\,|\overline B|=\overline n$. Set $m':=m-\overline m$ and $n':=n-\overline n$. For   $0\leq d\leq\min\{m,n\}$ if $m\ne n$ or $0\leq d<m=n$, we define
  \begin{align*}&\SylM_{d}(A,B)(x):= (-1)^{m'(m-d) }  \cdot \\& \cdot  \sum (-1)^{\sigma_{ R}} \frac{\mathcal{R}(A\backslash \overline A, \overline B\backslash B')\mathcal{R}(\overline A\backslash A',B\backslash B')\mathcal{R}(x,A')\mathcal{R}(x,B')}{ \mathcal{R}(A', \overline A\backslash A')\mathcal{R}(  B', \overline B\backslash B')}\,\cdot \\
 & \hskip2cm  \cdot \,  S_{d+1}^{(\widetilde R_1)}(A'\cup B'\cup \{x\})S_{m+n-d}^{(R_2)}((\overline{A}\backslash A')\cup B)  S_{m+n-d}^{(R_3)}(A\cup (\overline{B}\backslash B')),\end{align*}
where the sum is indexed by
\begin{itemize}
\item all possible disjoint unions $R_1\sqcup R_2\sqcup  R_3=\{1,\dots,m'+n' -d\}$ with
 $R_1\subset\{m+n-2d , \dots,m'+n' -d\}$, \
$|R_1| \le d-(\overline m+\overline n )+1$,
  $m'-d\le |R_2|\le m-d$ and
 $n'-d\le |R_3|\le n-d $,
\item all subsets $A'\subset \overline A, \,|A'|=|R_2|+ d - m' $,
\item all subsets $B'\subset \overline B, \,|B'|=|R_3|+\min\{ m',d-n'\},$
\end{itemize}
 $\sigma_{ R}$  is specified in \eqref{sign2} for $R:=(R_1, R_2, R_3)$, and $\widetilde R_1:=\{i-(m+n-2d-1):\, i\in R_1\}$.
\end{definition}
It is easy to verify that this notion generalizes Definition \ref{def-bigd}, since when $m'+n'\le d$, then $m'+n'-d\le 0$  which implies that the   sets $R_1, R_2$ and $R_3$ in the sum above are empty, and $|B'|=m'$. In this way, one recovers the previous multiple sum straightforwardly. On the other hand, when $m'+n'\ge d$, we have $\min\{ m',d-n'\}=d-n'$, and one can easily check that the three Schur polynomials are well defined, i.e. the submatrices of confluent Vandermonde matrices appearing in the formula are all square.

Furthermore, when $A$ is a set and we choose $\overline A=A$, that is $m'=0$, then $R_3=\{1,\dots, n'-d\}$, $R_1=R_2=\emptyset$, $|A'|=d$ and $|B'|=0$, and one can then check that
$\SylM_{d}(A,B)(x)= \Syl_{d}(A,B)(x)$.

An additional interesting feature of our formula is that when $B$ is a set instead of $A$, and we choose $\overline B=B$, that is $n'=0$, then $R_2=\{1,\dots,m'-d\}$, $R_1=R_3=\emptyset$,  $|A'|=0$ and $|B'|=\min\{m', d\}$. In this case one can check that for $m'\ge d$, one has
\begin{align*}\SylM_{d}(A,B)(x)&= (-1)^{d(m-1)} \sum_{B'\subset
B, \,|B'|=d}
\frac{\mathcal{R}(A,B\backslash
B')\,\mathcal{R}(x,B')}{\mathcal{R}(B',B\backslash B')}\nonumber \\
& =(-1)^{mn-d} \Syl_{d}(B,A)(x).\end{align*}
When $n'=0$ and $m'<d$, one can also prove that the same  identity holds, though it is not immediate and requires applying the Exchange Lemma described in Section~\ref{s2}. In any case,
the amazing fact is  that  $\SylM_{d}(A,B)(x)$ somehow ``recognizes" when $A$ {\em or} $B$ are sets.

The main result of our paper is the following generalization of Theorem \ref{sub-bigd-2}, which shows that $\SylM_{d}(A,B)(x)$ computes the subresultant in all the cases.

\begin{theorem}\label{alld2} Let $f,g\in \K[x]$  be monic polynomials of degrees $m$ and $n$, with multisets of roots $A$ and $B$.  Let    $\overline A$ and $\overline B$ be subsets of the sets of distinct roots of $f$ and $g$, respectively, and set $m':=m-|\overline A|$ and $ n':=n-|\overline B|$. Then for  $0\leq d\leq\min\{m,n\}$ if $m\ne n$ or $0\leq d<m=n$, we have
  {\begin{align*} &\Sres_d(f,g)(x)= (-1)^{d(m-d)} \SylM_{d}(A,B)(x).\end{align*}}
\end{theorem}

We  consider again Example~\ref{example1} to illustrate how  under Definition~\ref{SylM}, Theorem~\ref{alld2} indeed holds. The right-hand side $\SylM_d(A,B)(x)$ in this example is fully developed in Example~\ref{example4}  below.

 \begin{example} \label{example2} Take  $f=(x-a_1)(x-a_2)^2$ and $g=(x-b_1)^3$  associated to the multisets  $A=\{a_1,a_2,a_2\}$ with $\overline A=\{a_1,a_2\}$, $B=\{b_1,b_1,b_1\}$ with  $\overline B=\{b_1\}$ and $d=2$.  We have  $\Sres_2(f,g)(x)= g(x)-f(x)$ while in this case $\SylM_{2}(A,B)(x)$ equals

\begin{align*}&(a_2- b_1)(x-a_1)(x-a_2) - \frac{(a_2-b_1)(a_2-b_1) (x-a_1)(x-b_1)}{a_1-a_2}- \\  & \qquad - \,
 \frac{(a_1-b_1)(a_1-b_1) (x-a_2)(x-b_1)}{a_2-a_1} .\end{align*}
It is easy to check that the two expressions coincide.
\end{example}

Next, we also obtain analogous descriptions in term of roots for the B\'ezout coefficients $F_d(f,g)(x)$ and $G_d(f,g)(x)$ that appear when expanding
$$\Sres_d(f,g)(x)=F_d(f,g)(x)f(x)+G_d(f,g)(x)g(x),$$
with
\begin{equation*}
F_d(f,g)(x):= \det \begin{array}{|cccccc|c}
\multicolumn{6}{c}{\scriptstyle{m+n-2d}}&\\
\cline{1-6}
f_{m} & \cdots & &\cdots &f_{d+1-(n-d-1)} &x^{n-d-1}& \\
&  \ddots & &&\vdots  & \vdots& \scriptstyle{n-d}\\
& & f_m& \dots &f_{d+1}&1& \\
\cline{1-6}
g_{n} &\cdots & &\cdots  &g_{d+1-(m-d-1)}  &0&\\
&\ddots && &\vdots  &\vdots  &\scriptstyle{m-d}\\
&&g_{n} & \cdots &  g_{d+1} &0&\\
\cline{1-6} \multicolumn{2}{c}{}
\end{array}
\end{equation*}
and
\begin{equation*}
G_d(f,g)(x):= \det \begin{array}{|cccccc|c}
\multicolumn{6}{c}{\scriptstyle{m+n-2d}}&\\
\cline{1-6}
f_{m} & \cdots & &\cdots &f_{d+1-(n-d-1)} &0& \\
&  \ddots & &&\vdots  & \vdots& \scriptstyle{n-d}\\
& & f_m& \dots &f_{d+1}&0& \\
\cline{1-6}
g_{n} &\cdots & &\cdots  &g_{d+1-(m-d-1)}  &x^{m-d-1}&\\
&\ddots && &\vdots  &\vdots  &\scriptstyle{m-d}\\
&&g_{n} & \cdots &  g_{d+1} &1&\\
\cline{1-6} \multicolumn{2}{c}{}
\end{array}.
\end{equation*}
 Our new formulations  extend the following formulas in case of simple roots that already appear in
\cite[Art. 29]{Syl}  (see also \cite[Cor. 3.10]{KSV16}) for $0\le d< \min\{m,n\}$:
\begin{align}\label{Gdsets} F_{d}(f,g)(x)&=  (-1)^{m-d} \sum_{B'\subset B, |B'|=d+1}
\frac{\mathcal{R}(A,B\backslash B')\mathcal{R}(x,B\backslash B')}{\mathcal{R}(B',B\backslash B')},\nonumber \\
G_d(f,g)(x)&=(-1)^{d(m-d-1)}
  \sum_{ A^{\prime }\subset A,
|A^{\prime}|=d+1}\frac{\mathcal{R}(A\backslash A^{\prime
}, B) \mathcal{R}(x,A\backslash A^{\prime
})}{\mathcal{R}( A^{\prime},A\backslash A^{\prime })} .\end{align}

Our results are based on Lemma \ref{lemmaGd}, where we relate the  B\'ezout coefficients associated to $f$ and $g$ to
principal subresultants of bivariate polynomials in $\K[x,y]$ obtained from $f(y)$ and $g(y)$ by adding the variable $x$ to their roots. To our knowledge, this is the first result expressing the B\'ezout coefficients as special cases of subresultants.  This lemma allows us to use the results of  Theorems  \ref{sub-bigd-2} and \ref{alld2}  to get formulas for $F_d$ and $G_d$.

\begin{theorem}\label{Fd-bigd-2}
Let $f,g\in \K[x]$  be monic polynomials of degrees $m$ and $n$, with multisets of roots $A$ and $B$. Let $\overline A$ and $\overline B$ be subsets of the sets of distinct roots of $f$ and $g$, respectively, with $\overline m=|\overline A|$ and $\overline n=|\overline B|$, and  set $m':=m-\overline m$ and $ n':=n-\overline n$.
For any $d$ such that $m'+n'\leq d < \min\{m,n\}$, we have
  \small{\begin{align*} &F_d(f,g)(x)=\\
  & (-1)^{m-d+n'(\overline m-1)}\sum_{\substack{A'\subset \overline A\\|A'|=n' }}
\sum_{\substack{B'\subset \overline B\\|B'|=d+1-n' }} \frac{\mathcal{R}(A\backslash \overline A, \overline B\backslash B')\mathcal{R}(\overline A\backslash A',  B\backslash B')\mathcal{R}(x,\overline B\backslash B')}{ \mathcal{R}(A', \overline A\backslash A')\mathcal{R}(  B', \overline B\backslash B'),
}\\
&G_d(f,g)(x)=\\
&(-1)^{(d-m')(m-d-1)} \sum_{\substack{A'\subset \overline A\\|A'|=d+1-m' }}
\sum_{\substack{B'\subset \overline B\\|B'|=m' }} \frac{\mathcal{R}(\overline A\backslash A', B\backslash \overline B)\mathcal{R}( A\backslash A',  \overline B\backslash B')\mathcal{R}(x,\overline A\backslash A')}{ \mathcal{R}(A', \overline A \backslash A')\mathcal{R}(  B', \overline B\backslash B')
}.\end{align*}}x
\end{theorem}
These identities are particular cases of Theorem~\ref{Fd-alld}, which deals with   any value of $0\le d<\min\{m,n\}$.

\medskip
The paper is organized as follows: in Section \ref{s2} we describe the main ingredient in our proofs, Proposition~\ref{ApeJou}, which is a generalization of a result by F. Ap\'ery and J.-P. Jouanolou. In Section \ref{s3}, we apply this tool to justify the definition of $\SylM_{d}(A,B)(x)$ and show its connection with the subresultant.  For the sake of clarity, we first present our results for the case $d$ big enough and then in the following subsection,  we recall the definition  of confluent Schur polynomial and extend our definition and result to arbitrary values of $d$. Section~\ref{Fd} presents the formulas for the B\'ezout coefficients
$F_d(f,g)(x)$ and $G_d(f,g)(x)$.  We conclude in  Section~\ref{s4} by comparing our results with previous literature in the topic. The Appendix at the end contains the technical proofs of statements made in  Sections \ref{s2}, \ref{s3} and \ref{Fd}.

\medskip
A Maple code \cite{maple} computing the formulas described in Theorems \ref{bigd} and \ref{smalld} is freely available at\\
{\tt http://cms.dm.uba.ar/Members/mvaldettaro/code.mw} \\
This code has been used for computing most of the examples which illustrate this paper.

\bigskip
\begin{ac} We thank the referee for suggesting to give more motivating examples and reorganizing the paper.  We also would like to thank the Department  of Mathematics at the University of  Buenos Aires for their hospitality, where the work on this project started during a visit of Carlos D'Andrea and Agnes Szanto in 2015. We would also like to thank the Department de Matem\`atiques i Inform\`atica  at the University of Barcelona for hosting Marcelo Valdettaro to work on this project, which is part of his PhD thesis \cite{Val17}, in July 2016. Carlos D'Andrea acknowledges financial support from the Spanish Ministry of Economy and Competitiveness, through the ``Mar\'ia de Maeztu'' Programme for Units of Excellence in R \& D (MDM-2014-0445) and the Research Project MTM2013-40775-P,
Teresa Krick and Marcelo Valdettaro were partially supported by  CONICET PIP-11220130100073 and UBACyT 2014-2017-20020130100143BA, and Agnes Szanto was partially supported by NSF grants CCF-1217557 and CCF-1813340. D'Andrea, Krick and Valdettaro were also partially supported by ANPCyT PICT-2013-0294. \end{ac}

\section{A generalization of a result by Ap\'ery \& Jouanolou}\label{s2}

This section deals with a generalization of a result by Ap\'ery and Jouanolou that appears in~\cite[Prop.91]{ApJo06}, which is quite surprising and seems of independent interest. No multisets are involved here.

\begin{proposition}\label{ApeJou} Let $A,B\subset \K$ be finite sets with $|A|=m$, $|B|=n$. Set $0\le d\le m$.  Let $X$ be a set of variables  and  $E\subset \K$ be any finite set satisfying
$$|E|\ge \max\{ |X|+d, m+n-d, m\}.$$ Then
\begin{align*}&\sum_{\substack{A_1\sqcup A_2=A\\
|A_1|=d,\,|A_2|=m-d}} \frac{\mathcal{R}(A_2,B)\mathcal{R}(X,A_1)}{\mathcal{R}(A_1,A_2)} = \\ &\qquad \qquad =\sum_{\substack{E_1\sqcup E_2\sqcup E_3 = E\\
|E_1|=d,|E_2|=m-d,|E_3|=|E|-m}}\frac{\mathcal{R}(A,E_3)\mathcal{R}(E_2,B)\mathcal{R}(X,E_1)}
{\mathcal{R}(E_1,E_2)\mathcal{R}(E_1,E_3)\mathcal{R}(E_2,E_3)}. \end{align*}
\end{proposition}

The original result in \cite[Prop.91]{ApJo06} states that for $|E|=m+n-d$ and $d\le \min\{m,n\}$ for $m\ne n$ or $d<m=n$, one has
\begin{align*}\Sres_d(f,g)(x)= \sum_{\substack{E_1\sqcup E_2\sqcup E_3 = E\\
|E_1|=d,|E_2|=m-d,|E_3|=n-d}}\frac{\mathcal{R}(E_3,A)\mathcal{R}(E_2,B)\mathcal{R}(x,E_1)}
{\mathcal{R}(E_2,E_1)\mathcal{R}(E_3,E_1)\mathcal{R}(E_3,E_2)}.\end{align*}
This is a particular case of our result by~\eqref{sumasimple} and the definition of the  sum~\eqref{defsinglesum}.

We illustrate the result with a toy example, which shows how the symmetric interpolation developed by W.Y. Chen and J.D. Louck \cite{CL96} (or see  Proposition~\ref{Lagr} below)   applies here, and leave its technical proof to Section~\ref{s51} in the Appendix.

\begin{example} Take $A=\{a_1,a_2\}, \, B=\{b\}$, and $d=1$.  For $X=\{x\}$ we have
\begin{align*}\sum_{\substack{A_1\sqcup A_2=A\\
|A_1|=1,\,|A_2|=1}} \frac{\mathcal{R}(A_2,B)\mathcal{R}(x,A_1)}{\mathcal{R}(A_1,A_2)} & = \frac{(a_1-b)(x-a_2)}{a_2-a_1} + \frac{(a_2-b)(x-a_1)}{a_1-a_2}\\
&= -(x-b).\end{align*}
On the other side, for $E=\{e_1,e_2\}$  we have
\begin{align*}&\sum_{\substack{E_1\sqcup E_2\sqcup E_3 = E\\
|E_1|=1,|E_2|=1,|E_3|=0}}\frac{\mathcal{R}(A,E_3)\mathcal{R}(E_2,B)\mathcal{R}(x,E_1)}
{\mathcal{R}(E_1,E_2)\mathcal{R}(E_1,E_3)\mathcal{R}(E_2,E_3)}\\ & \qquad =
\frac{(e_1-b)(x-e_2)}{e_2-e_1} + \frac{(e_2-b)(x-e_1)}{e_1-e_2}\ = \ -(x-b)\ \mbox{ as well},
\end{align*}
which is obvious in this case, or can be easily checked for instance by Lagrange interpolation in $e_1$ and $e_2$. This is an example of Ap\'ery and Jouanolou's result.\\
Now consider $X=\{x_1,x_2\}$, $E=\{e_1,e_2,e_3\} $, and $A$, $B$ as above. On one hand we have
\begin{align*}\sum_{\substack{A_1\sqcup A_2=A\\
|A_1|=1,\,|A_2|=1}} &\frac{\mathcal{R}(A_2,B)\mathcal{R}(X,A_1)}{\mathcal{R}(A_1,A_2)}\\ & = \frac{(a_1-b)(x_1-a_2)(x_2-a_2)}{a_2-a_1} + \frac{(a_2-b)(x_1-a_1)(x_2-a_1)}{a_1-a_2}\\
&=-x_1x_2+b(x_1+x_2)+a_1a_2-b(a_1+a_2) \ =: \ f(x_1,x_2).\end{align*}
On the other hand, for $E=\{e_1,e_2,e_3\}$,
$$\sum_{\substack{E_1\sqcup E_2\sqcup E_3 = E\\
|E_1|=1,|E_2|=1,|E_3|=1}}\frac{\mathcal{R}(A,E_3)\mathcal{R}(E_2,B)\mathcal{R}(X,E_1)}
{\mathcal{R}(E_1,E_2)\mathcal{R}(E_1,E_3)\mathcal{R}(E_2,E_3)},$$
is a symmetric polynomial $g$ in 2 variables  of multidegree bounded by 1, which, by symmetric interpolation (see Proposition~\ref{Lagr} below), is determined by its value on all subsets of size 2 of $E=\{e_1,e_2,e_3\}$.
Let us check that the  symmetric polynomial $g$ agrees with the above symmetric polynomial $f$ in all subsets of size 2 of $E$: for $1\le i<j\le 3$, one has
 \begin{align*} g(e_i,e_j)&= \frac{(a_1-e_i)(a_2-e_i)(e_j-b)}{e_j-e_i}+ \frac{(a_1-e_j)(a_2-e_j)(e_i-b)}{e_i-e_j}\\ &= -e_ie_j+b(e_i+e_j)+a_1a_2-b(a_1+a_2)= f(e_i, e_j).\end{align*}
 Thus, $f=g$.\\
Finally, since the two bivariate polynomials $f$ and $g$ coincide, their leading coefficient with respect to $x_2$ also coincide, which implies that
\begin{align*}\sum_{\substack{A_1\sqcup A_2=A\\
|A_1|=1,\,|A_2|=1}} &\frac{\mathcal{R}(A_2,B)\mathcal{R}(x_1,A_1)}{\mathcal{R}(A_1,A_2)} \\ = &\sum_{\substack{E_1\sqcup E_2\sqcup E_3 = E\\
|E_1|=1,|E_2|=1,|E_3|=1}}\frac{\mathcal{R}(A,E_3)\mathcal{R}(E_2,B)\mathcal{R}(x_1,E_1)}
{\mathcal{R}(E_1,E_2)\mathcal{R}(E_1,E_3)\mathcal{R}(E_2,E_3)}\end{align*}
as well. This means that the equality also holds for one variable $x$, which is an example of a case where $|E|>\max\{|X|+d,m+n-d,m\}$.

\end{example}

A prominent consequence of Proposition~\ref{ApeJou} is that the sum in its right-hand side, since it coincides with the sum in the left-hand side,   does not depend on the particular choice of the set $E$,  as soon as it is large enough, but only on the sets $A$ and $B$. This enables us to compute it by any suitable specialization of the set $E$.
\section{Application to subresultants}\label{s3}

This section is devoted to motivate Definitions~\ref{def-bigd} and \ref{SylM}, and prove Theorems~\ref{sub-bigd-2} and \ref{alld2} of the introduction. This will be done via Theorems~\ref{bigd} and \ref{smalld} below, where $A$ and $B$ are assumed to be sets instead of multisets, and $\overline A$, $\overline B$ are arbitrary subsets of $A$, $B$ respectively.  Proposition~\ref{ApeJou}, which can be interpreted as a multivariate version of $\Syl_{d}(A,B)(x)$ by means of an arbitrary auxiliary set $E$ (where only the size of $E$ matters),  allows us to specialize $E$ on sets in such a way that the denominators only depend on these $\overline A$ and $\overline B$. Then, in the proofs of Theorems ~\ref{sub-bigd-2} and \ref{alld2}, we let the elements of $A$ or $B$ collide, and our  formulas remain well defined as long as we assume that the elements of $\overline A $ and $\overline B$ are all distinct.

We start with the easier case of  $d$ sufficiently large to be in the range of Definition~\ref{def-bigd}.

\subsection{The case of $d$ sufficiently large} {\ }

\begin{theorem} \label{bigd} Let $A,B\subset \K$  be sets with $|A|=m$ and $|B|=n$.  Let $\overline{A}\subseteq A$ and $\overline{B}\subseteq B$ be any non-empty  subsets of $A$ and $B$ respectively,  with $|\overline{A}|=\overline{m},\,|\overline{B}|=\overline{n},$ and set $m':=m-\overline{m},\, n':=n-\overline{n}$. Assume that $d$ satisfies  $m'+n'\leq d\leq\min\{m,n\},$  and let $X$ be a set of variables with $|X|\le m+n-2d$. Then  { \begin{align*} &\sum_{\substack{A_1\sqcup A_2 = A\\
|A_1|=d,\,|A_2|=m-d}}\frac{\mathcal{R}(A_2,B)\mathcal{R}(X,A_1)}{\mathcal{R}(A_1,A_2)}= (-1)^{m'(m-d)}\,\cdot \\& \ \ \cdot \sum_{\substack{A'\subset \overline A\\|A'|=d-m' }}
\sum_{\substack{B'\subset \overline B\\|B'|=m' }} \frac{\mathcal{R}(A\backslash \overline A, \overline B\backslash B')\mathcal{R}(\overline A\backslash A', B\backslash B')\mathcal{R}(X,A')\mathcal{R}(X,B')}{ \mathcal{R}(A', \overline A\backslash A')\mathcal{R}(  B', \overline B\backslash B')
}.\end{align*}}

\end{theorem}

\begin{proof}
We first assume that $A\cap B=\emptyset$. By Corollary~\ref{ApeJou} applied to $E:=\overline A\cup \overline B$, with  $|E|=\overline m + \overline n\ge m+n-d$ by assumption, we have
\[\sum_{\substack{A_1\sqcup A_2 = A\\
|A_1|=d,\\|A_2|=m-d}}\frac{\mathcal{R}(A_2,B)\mathcal{R}(X,A_1)}{\mathcal{R}(A_1,A_2)}=\sum_{\substack{E_1\sqcup E_2\sqcup E_3 = \overline A \cup \overline B\\
 |E_1|=d, |E_2|=m-d\\
 |E_3|=\overline m + \overline n-m}}\frac{\mathcal{R}(A,E_3)\mathcal{R}(E_2,B)\mathcal{R}(X,E_1)}{\mathcal{R}(E_1,E_2)
 \mathcal{R}(E_1,E_3)
 \mathcal{R}(E_2,E_3)}.\]
Now, if  $A\cap E_3\neq \emptyset$ then $ \mathcal{R}(A,E_3)=0 $  and if $E_2\cap B\neq\emptyset$ then $\mathcal{R}(E_2,B)=0$. Therefore, in each non-zero summand on the right hand side we have  $E_3\subset \overline B$ and $E_2\subset \overline A$. Setting $A'=\overline A \backslash E_2$ and $B'=\overline B\backslash E_3$, we get that $E_3=\overline B \backslash B'$, $E_2=\overline A\backslash A'$ and $E_1=A'\cup B' $, and therefore we can rewrite the right hand side as
{\small{\begin{align*}\label{inicio} &\sum_{\substack{A'\subset \overline A\\|A'|=d-m' }}
\sum_{\substack{B'\subset \overline B\\|B'|=m' }} \frac{\mathcal{R}(A,\overline B\backslash B')\mathcal{R}(\overline A\backslash A',B)\mathcal{R}(X,A')\mathcal{R}(X,B')}{ \mathcal{R}(A'\cup B', \overline A\backslash A')\mathcal{R}( A'\cup B', \overline B\backslash B')\mathcal{R}(\overline A\backslash A', \overline B\backslash B')
}\\ \nonumber
& \quad
= \sum_{\substack{A'\subset \overline A\\|A'|=d-m' }}
\sum_{\substack{B'\subset \overline B\\|B'|=m' }} \frac{\mathcal{R}(A,\overline B\backslash B')\mathcal{R}(\overline A\backslash A',B)\mathcal{R}(X,A')\mathcal{R}(X,B')}{ \mathcal{R}(A', \overline A\backslash A')\mathcal{R}( B', \overline A\backslash A')\mathcal{R}(\overline A , \overline B\backslash B')\mathcal{R}( B', \overline B\backslash B')
}
\\ \nonumber
& \quad
= (-1)^{|B'|\cdot |\overline A\backslash A'|}\sum_{\substack{A'\subset \overline A\\|A'|=d-m' }}
\sum_{\substack{B'\subset \overline B\\|B'|=m' }} \frac{\mathcal{R}(A\backslash \overline A, \overline B\backslash B')\mathcal{R}(\overline A\backslash A',B\backslash B')\mathcal{R}(X,A')\mathcal{R}(X,B')}{ \mathcal{R}(A', \overline A\backslash A')\mathcal{R}(  B', \overline B\backslash B')
}\end{align*}}}
as desired, since $|B'|\cdot |\overline A\backslash A'|= m'(m-d)$.

The general statement follows from the fact that the two expressions generically coincide.
\end{proof}

We note that the right-hand side of the equality in Theorem~\ref{bigd}  makes sense even when $A,B$  are multisets instead of sets, for one only needs $\overline A$, $\overline B$ to be sets. For $X=\{x\}$ we can then define, as stated in Definition~\ref{def-bigd},  the notion of  Sylvester sum for multisets $A$ and $B$ and $d$ within the bounds of Theorem~\ref{bigd}, which extends the usual notion of Sylvester  sums for sets.

\begin{proof}[Proof of Theorem \ref{sub-bigd-2}] Since neither $\SylM_d(A,B)$  nor $\Sres_d(f,g)$ depends on the ordering of the elements in $A$ and $B$, we can assume without loss of generality that the distinct elements of $\overline A$ and $\overline B$ appear as the first $\overline m$ and $\overline n$ elements of $A$ and $B$, respectively.   Define sets of indeterminates $Y=\{ y_1, \ldots, y_m\}$  and $Z=\{z_1,, \ldots, x_n\}$, and set $f^Y=\prod_{i=1}^m(x-y_i)$ and $g^Z=\prod_{i=1}^n (x-z_i)$. Then if we set  $\overline Y:=\{y_1, \ldots, y_{\overline m}\}\subset Y$ and $\overline Z =\{z_1, \ldots, z_{\overline n}\}\subseteq Z$ and $m'+n''\le d\leq\min\{m,n\}$ if $m\ne n$ or $m'+n'\leq d<m=n$, according to Theorem~\ref{bigd} we have
{\small\begin{align*} &\Syl_{d}(Y,Z)(x)= \\&(-1)^{m'(m-d)}\!\! \sum_{\substack{Y'\subset \overline Y\\|Y'|=d-m' }}
\sum_{\substack{Z'\subset \overline Z\\|Z'|=m' }} \frac{\mathcal{R}(Y\backslash \overline Y, \overline Z\backslash Z')\mathcal{R}(\overline Y\backslash Y', Z\backslash Z')\mathcal{R}(x,Y')\mathcal{R}(x,Z')}{ \mathcal{R}(Y', \overline Y\backslash Y')\mathcal{R}(  Z', \overline Z\backslash Z')
}.\end{align*}}
On the other hand, by~\eqref{sumasimple},
$\Sres_d(f^Y,g^Z)(x)= (-1)^{d(m-d)}\Syl_{d}(Y,Z)(x).$
Therefore, for $d$ within the stated bounds,
{\small \begin{align*} &\Sres_d(f^Y,g^Z)(x)=\\& (-1)^{(d-m')(m-d)}\!\! \sum_{\substack{Y'\subset \overline Y\\|Y'|=d-m' }}
\sum_{\substack{Z'\subset \overline Z\\|Z'|=m' }} \frac{\mathcal{R}(Y\backslash \overline Y, \overline Z\backslash Z')\mathcal{R}(\overline Y\backslash Y', Z\backslash Z')\mathcal{R}(x,Y')\mathcal{R}(x,Z')}{ \mathcal{R}(Y', \overline Y\backslash Y')\mathcal{R}(  Z', \overline Z\backslash Z')
}.\end{align*}}
We end the proof by setting $y_{1}\to a_1,\dots, y_{\overline m}\to a_{\overline m}, \ldots, y_m\to a_m$, $z_{1}\to b_1,\dots, z_{\overline n}\to b_{\overline n},\ldots, z_n\to b_n$  noting that both sides of the equality are well-defined after this specialization.
\end{proof}


\subsection{The general case.}\label{s44}

In order to deal with the situation where $0\le d<m'+n'$ we need to recall the definition of Schur polynomials. Given a partition $$\lambda = (\lambda_1, \lambda_2, \dots,\lambda_r), \ \lambda_i\in \Z_{\ge 0}\mbox{ for } 1\le i\le r, \mbox{ with } \lambda_1\ge \lambda_2\ge \cdots \ge \lambda_r,$$  the Schur polynomial
$s_\lambda(X)$ for a set $X=\{x_1,\dots,x_r\}$  is defined as the ratio

$$s_\lambda(X)= \frac{\det \left(\begin{array}{cccc}x_1^{\lambda_1+r-1}&x_2^{\lambda_1+r-1}& \cdots & x_r^{\lambda_1+r-1}\\x_1^{\lambda_2+r-2}&x_2^{\lambda_2+r-2}& \cdots & x_r^{\lambda_2+r-2} \\
\vdots & \vdots & \ddots & \vdots \\
x_1^{\lambda_r}&x_2^{\lambda_r}& \cdots & x_r^{\lambda_r}
\end{array}\right)}{\det\left(\begin{array}{ccc}x_1^{r-1}& \dots & x_r^{r-1}\\
\vdots & & \vdots \\
1& \dots & 1\end{array}\right)}.$$

That is, Schur polynomials are ratios of subdeterminants of Vandermonde matrices, where in the numerator some rows of a regular Vandermonde matrix are deleted, while in the denominator  a regular Vandermonde matrix occurs. Note that Schur polynomials are symmetric in $x_1,\dots,x_r$, and thus it makes sense to write
$s_\lambda(X)$ for a set $X$.
For convenience here, we will not follow this usual notation for Schur polynomials given by partitions but introduce a notation with a set of exponents as follows: for $k,r\in\N,\,k\geq r,$ we set
$$V_k(X)=\left(\begin{array}{ccc}x_1^{k-1} & \dots & x_r^{k-1}\\
\vdots & & \vdots\\1& \dots & 1\end{array}\right)$$ to be  the regular rectangular  Vandermonde matrix of size $k\times r.$ When $k=r$ we write $V(X)$ for simplicity. For a subset of row indexes $R=\{i_1,\dots,i_{k-r}\}\subset \{1,\dots, k\},$  we will denote by $V_k^{(R)}(X)$ the square submatrix of $V_k(X)$ obtained by removing the rows  indexed by $R.$ We then define
\begin{equation}\label{schur}S_k^{(R)}(X):=\frac{\det(V_k^{(R)}(X))}{\det(V(X))},\end{equation}
that is $S_k^{(R)}(X)$ is the Schur polynomial  associated to the set of indexes $\{1,\ldots, k\}\setminus R.$

In a more general setting, if $X=\{\underbrace{x_1,\dots,x_{1}}_{j_1},\dots,\underbrace{x_{\overline r},\dots,x_{\overline r}}_{j_{\overline r}}\}$ is a multiset  with $r=j_1+\cdots + j_{\overline r}$, we define a generalized or confluent Vandermonde matrix  instead of the regular Vandermonde matrix of size $k\times r$ as (c.f.  \cite{Kal1984})
$$V_k(X)=\left(\begin{array}{ccc}
 V_{k}(x_1,j_1) & \dots &V_{k}(x_{\overline r},j_{\overline r})
\end{array}\right)$$
where for any $j$,
 $V_{k}(x_i, j)$ of size $k\times j$ is defined by
$$
V_{k}(x_i, j):=\left(\begin{array}{ccccc}
x_i^{k-1} & (k-1)x_i^{k-2} &(k-1)(k-2)x_i^{k-3}&\dots & \frac{(k-1)!}{(k-j )!}x_i^{k-j}\\
\vdots & \vdots& \vdots& &\vdots\\
x^2_i&2x_i &2&\dots & 0 \\
x_i&1 &0&\dots & 0 \\
1&0 &0&\dots & 0
\end{array}\right),$$
where when $k=r$ one writes again $V(X)$ for simplicity. It is known  that $V(X)$ is invertible when $x_i\ne x_j$ for $i\ne j$.

Then one can define  confluent Schur polynomials in the same way  as before:
let $R=\{i_1,\dots,i_{k-r}\}\subset \{1,\dots, k\}$  be a subset of the row indexes, then we will denote by $V_k^{(R)}(X)$ the square submatrix of $V_k(X)$ obtained by removing from it the rows  indexed by $R$,  and define
\begin{equation}\label{schg}
S_k^{(R)}(X):=\frac{\det(V_k^{(R)}(X))}{\det(V(X))}.
\end{equation}

Note that in principle $S_k^{(R)}(X)$ is  a rational function, and it may not be defined over fields of positive small characteristic.  The next result shows that it is actually a polynomial, and hence its definition can be done over any field $\K$.
\begin{lemma}\label{limit}
$S_k^{(R)}(X)$ is a symmetric polynomial in the $X$-variables with coefficients in $\K.$
\end{lemma}
\begin{proof}
When $X$ is a set instead of being a multiset, the Schur function defined in~\eqref{schg} coincides with the Schur polynomial  defined in~\eqref{schur}, so the claim obviously holds in this situation.
To prove the statement in the general case, consider a set $X=\{x_{1,1},\dots, x_{1,j_1},\dots, x_{\overline r, 1},\dots, x_{\overline r, i_{\overline r}}\}$ which will ``converge''  to a  multiset $Y$ by setting $x_{1,i}\to y_1$ for $1\le i\le j_1$, \dots, $x_{\overline r,i}\to y_{\overline r}$ for $ 1\le i\le j_{\overline r}$. Then
$$S_k^{(R)}(X) \to S_k^{(R)}(Y)$$
as it can be seen for instance by computing the limits for   $x_{1,2}\to x_{1,1}$ by L'H\^opital rule,   for $x_{i,3}\to x_{i,1}$ if necessary, and repeating the same for the other terms $x_{k,2}\to x_{k,1}$, etc.
This shows that  $S_k^{(R)}(Y)\in\K[Y].$
\end{proof}

For a given  (increasingly ordered) set  $R\subset\{1,\ldots, r\}$, we set
$\sg_r(R):=(-1)^\sigma
$, where $\sigma$ is a number of transpositions needed to move this set  to the first positions in $\{1,\dots, r\}$, i.e. if $R=\{i_1,\dots i_{s}\}$ with $1\le i_1<\cdots < i_{s}\le r$, then
$\sigma $ is the parity of the number of transpositions needed to bring $(1,\dots,r)$ to $(i_1,\dots i_{s}, \dots)$, without changing the   relative order of the other elements.

Also, for a given partition $\{1,\ldots, r\}=R_1\sqcup R_2\sqcup \ldots\sqcup R_\ell$, with  $R:=(R_1, \ldots R_\ell)$ we denote $\sg({ R})=(-1)^\sigma,$ where $\sigma$ is  the parity of the number of transpositions needed to bring the ordered set
$(R_1,\ldots, R_\ell)$ (we assume that each of them is also increasingly ordered) to $\{1,\ldots, r\}.$

\begin{theorem} \label{smalld} Let $A,B\subset \K$  be sets with $|A|=m$ and $|B|=n$.  Let $\overline{A}\subseteq A$ and $\overline{B}\subseteq B$ be any non-empty subsets of $A$ and $B$ respectively, with $|\overline{A}|=\overline{m}$ and $|\overline{B}|=\overline{n}$ and set $m':=m-\overline{m}$ and $n':=n-\overline{n}$. Assume that $0\le d\le \min\{m,n\}$ if $m\ne n$ or  $0\le d<m=n$ satisfies in addition $d<m'+n'$.  Then:

 (1) If $0\le d<\overline m+\overline n$ then

\begin{align*}&\sum_{\substack{A_1\sqcup A_2 = A\\
|A_1|=d,\\|A_2|=m-d}}\frac{\mathcal{R}(A_2,B)\mathcal{R}(x,A_1)}{\mathcal{R}(A_1,A_2)}\,=(-1)^{m'(m-d)}\,\sum_{\substack{R_2\sqcup R_3=\{1,\dots, m' + n' -d \}\\
|R_2|=r_2,\, m'-d\le r_2\le m-d\\ |R_3|=r_3,\,n'-d\le r_3\le n-d}}(-1)^{\sigma_{R}}\ \cdot \\[1mm]
& \qquad \cdot \sum_{\substack{A'\subset \overline A\\|A'|=r_2-(m' -d)}}
\sum_{\substack{B'\subset \overline B\\|B'|=r_3-(n' -d)}}    \frac{\mathcal{R}(A\backslash \overline A, \overline B\backslash B')\mathcal{R}(\overline A\backslash A',B\backslash B')\mathcal{R}(x,A')\mathcal{R}(x,B')}{ \mathcal{R}(A', \overline A\backslash A')\mathcal{R}(  B', \overline B\backslash B')}\,\cdot \\
 & \hskip5cm \cdot S_{m+n-d}^{(R_2)}((\overline{A}\backslash A')\cup B) S_{m+n-d}^{(R_3)}(A\cup (\overline{B}\backslash B')),\end{align*}
 where for the partition $R_2\sqcup R_3=\{1,\dots, m'+n'-d\}$ and $R=(R_2, R_3)$  \begin{equation*}(-1)^{\sigma_{ R}}:=(-1)^{  r_2(\overline m -1) + r_3(m' +n' -d-1) +r_2r_3}\sg({R}).\end{equation*}
(2)  If $\overline m+\overline n\le d<m'+n'$,
{\small \begin{align*}&\sum_{\substack{A_1\sqcup A_2 = A\\
|A_1|=d,\\|A_2|=m-d}}\frac{\mathcal{R}(A_2,B)\mathcal{R}(x,A_1)}{\mathcal{R}(A_1,A_2)}\, = \,(-1)^{m'(m-d)}
 \sum_{\substack{R_1\sqcup R_2\sqcup  R_3=\{1,\dots,m'+n'-d\}\\
 R_1\subset\{m+n-2d , \dots,m'+n' -d\}, \\| R_1|=r_1, \, r_1\le d-(\overline m+\overline n )+1\\ |R_2|=r_2, \,m'-d \le r_2\le m-d\\ |R_3|=r_3,\, n'-d\le r_3\le n-d }} (-1)^{\sigma_{ R}}\,\cdot \\[1mm]
 & \qquad  \cdot  \sum_{\substack{A'\subset \overline A\\|A'|=r_2-(m' -d)}}
\sum_{\substack{B'\subset \overline B\\|B'|=r_3-(n' -d)}}  \frac{\mathcal{R}(A\backslash \overline A, \overline B\backslash B')\mathcal{R}(\overline A\backslash A',B\backslash B')\mathcal{R}(x,A')\mathcal{R}(x,B')}{ \mathcal{R}(A', \overline A\backslash A')\mathcal{R}(  B', \overline B\backslash B')}\,\cdot \\
 &  \hskip3cm  \cdot  S_{d+1}^{(\widetilde R_1)}(A'\cup B'\cup x)S_{m+n-d}^{(R_2)}((\overline{A}\backslash A')\cup B)S_{m+n-d}^{(R_3)}(A\cup (\overline{B}\backslash B')),\end{align*}}
where  $\widetilde R_1:=\{i-(m+n-2d-1):\, i\in R_1\},$ and
for the partition $R_1\sqcup R_2\sqcup R_3=\{1,\dots, m'+n'-d\}$ and $R=(R_1, R_2, R_3)$   \begin{equation}\label{sign2}(-1)^{\sigma_{ R}}:=(-1)^{  r_1(n-d+r_2+r_3) + r_2(\overline m -1) + r_3(m' +n' -d-1) +r_2r_3}\sg({ R}).\end{equation}
\end{theorem}

We leave the proof of Theorem~\ref{smalld} to Section~\ref{s52} in the Appendix. Here we illustrate this theorem by working out the full details the case corresponding to Example~\ref{example2} of the introduction.

\begin{example} \label{example4} Let $A=\{a_1,a_2,a_3\},\,\overline A=\{a_1,a_2\}$, $B=\{b_1,b_2,b_3\},\,\overline B=\{b_1\}$, and $d=2$. Set $f=(x-a_1)(x-a_2)(x-a_3)$ and $g=(x-b_1)(x-b_2)(x-b_3)$.
On one side we have
\begin{align*}&\sum_{\substack{A_1\sqcup A_2 = A\\
|A_1|=2,|A_2|=1}}  \frac{\mathcal{R}(A_2,B)\mathcal{R}(x,A_1)}{\mathcal{R}(A_1,A_2)} \ = \ \frac{g(a_1)(x-a_2)(x-a_3)}{(a_2-a_1)(a_3-a_1)} \,+ \\& \qquad  +\, \frac{g(a_2)(x-a_1)(x-a_3)}{(a_1-a_2)(a_3-a_2)} + \frac{g(a_3)(x-a_1)(x-a_2)}{(a_1-a_3)(a_2-a_3)} \, = \, g(x)-f(x),\end{align*}
since, by Lagrange interpolation, these two polynomials of degree $\le 2$  agree on $a_1$, $a_2$ and $a_3$. \\
Now, since $d\le \overline m +\overline n$ in this case, we need to compute the right-hand side in (1) above.
First, $m'(m-d)=1$. Also,  $R_2\sqcup R_3=\{1\}$. For $R_2=\{1\}$, $r_2=1$, $|A'|=1$ , and $R_3=\emptyset$,  $r_3=0$,  $|B'|=0$, $(-1)^{\sigma_R}=-1$, and one can check that $ S_{4}^{(1)}(\emptyset \cup B) =1$ and $ S_{4}^{(\emptyset)}(A\cup \{b_1\})=1$. This gives the term $$(a_3- b_1)(x-a_1)(x-a_2).$$
Similarly, for $r_2=0$ and $r_3=1$, $|A'|=1$ and $|B'|=1$. The computation gives two terms depending on $A'=\{a_1\}$ and $A'=\{a_2\}$:
$$- \frac{(a_2-b_2)(a_2-b_3) (x-a_1)(x-b_1)}{a_1-a_2}
 - \frac{(a_1-b_2)(a_1-b_3) (x-a_2)(x-b_1)}{a_2-a_1}.$$
 So the final sum equals
\begin{align*}&(a_3- b_1)(x-a_1)(x-a_2)- \frac{(a_2-b_2)(a_2-b_3) (x-a_1)(x-b_1)}{a_1-a_2}-\\  & \qquad - \, \frac{(a_1-b_2)(a_1-b_3) (x-a_2)(x-b_1)}{a_2-a_1}.\end{align*}
One can easily verify that this expression coincides with $g(x)-f(x)$  by interpolating in $a_1,a_2$ and $b_1$ for instance. The expression in Example~\ref{example2} is obtained by setting $a_2= a_3$ and $b_1=b_2=b_3$.
\end{example}

We are ready now to conclude the proof of  Theorem \ref{alld2}.

\begin{proof}[Proof of Theorem \ref{alld2}]
First we  note that $\SylM_{d}(A,B)(x)$ introduced in Definition~\ref{SylM} not only  generalizes Definition~\ref{def-bigd} as mentioned in the introduction, but also generalizes the  term in the right-hand side of Theorem~\ref{smalld}(1) for sets, since when  $d<\overline m + \overline n$, $R_1\subset \{ m+n-2d,\dots, m'+n'-d\} =\emptyset $. Therefore, thanks to Identity~\eqref{sumasimple}, Theorems~\ref{bigd} and \ref{smalld}, one has that the following equality holds for sets $A$ and $B$, any subsets $\overline A\subset A$ and $\overline B\subset B$ and any $0\le d\le \min\{m,n\}$ if $m\ne n$ or $0\le d<m=n$:
$$\Sres_d(f,g)(x)=(-1)^{d(m-d)} \SylM_{d}(A,B)(x).$$
The transition from sets to multisets is then straightforward by taking limits of sets to multisets, as in the proof of Theorem~\ref{sub-bigd-2}, thanks to Lemma~\ref{limit} and its proof, since  both quantities are  well-defined for multisets.
\end{proof}


\section{Application to the B\'ezout coefficients}\label{Fd}

We first show  an interesting new connection between the B\'ezout coefficients $F_d(f,g)(x)$ and $G_d(f,g)(x)$ and  the order $d+1$ principal subresultants of bivariate polynomials obtained from  $f$ and $g$.

\begin{lemma}\label{lemmaGd} Let $f,g\in \K[x]$  be  polynomials of degrees $m$ and $n$, respectively, and define  $\tilde f=f(y)\cdot (y-x)\in \K(x)[y]$ and $\tilde g=g(y)\cdot (y-x)\in \K(x)[y]$. Then for any $0\le d<\min\{m,n\}$, we have
\begin{align*}
F_d(f,g)(x)&=(-1)^{m-d}\coeff_{y^{d+1}}\big(\Sres_{d+1}(\tilde f, g(y))\big),\\
G_d(f,g)(x)&=\coeff_{y^{d+1}}\big(\Sres_{d+1}(f(y),\tilde g)\big).\end{align*}
\end{lemma}
\begin{proof} We first consider the case when $f$ and $g$ have distinct roots $A$ and $B$.
By Identity~\eqref{Gdsets} we have
\begin{align*}G_d(f,g)(x)&=(-1)^{d(m-d-1)}
  \sum_{ A^{\prime }\subset A,
|A^{\prime}|=d+1}\frac{\mathcal{R}(A\backslash A^{\prime
}, B) \mathcal{R}(x,A\backslash A^{\prime
})}{\mathcal{R}( A^{\prime},A\backslash A^{\prime })} \\
&= (-1)^{(d+1)(m-d-1)}
  \sum_{ A^{\prime }\subset A,
|A^{\prime}|=d+1}\frac{\mathcal{R}(A\backslash A^{\prime
}, B\cup \{x\} )}{\mathcal{R}( A^{\prime},A\backslash A^{\prime })} \\
&=(-1)^{(d+1)(m-d-1)}\coeff_{y^{d+1}}\Big(\sum_{ A^{\prime }\subset A,
|A^{\prime}|=d+1}\frac{\mathcal{R}(A\backslash A^{\prime
}, B\cup \{x\} )\mathcal{R}(y, A')}{\mathcal{R}( A^{\prime},A\backslash A^{\prime })}\Big)\\
&= (-1)^{(d+1)(m-d-1)}\coeff_{y^{d+1}}\big(\Syl_{d+1}(A,B\cup\{x\})(y)\big)\\&=
\coeff_{y^{d+1}}\big(\Sres_{d+1}(f(y),\tilde g)\big)
\end{align*}
and
\begin{align*}
F_d(f,g)(x)&=(-1)^{(m-d)(n-d)} G_d(g,f)(x)\\
&= (-1)^{(m-d)(n-d)}\coeff_{y^{d+1}}\Big(\Sres_{d+1}(g(y),\tilde f)\Big)\\&=
(-1)^{m-d} \coeff_{y^{d+1}}\Big(\Sres_{d+1}(\tilde f,g(y))\Big).
\end{align*}
The identities for arbitrary polynomials $f$ and $g$ follow by continuity, since all expressions are well-defined in case of multiple roots.
\end{proof}
Now we are ready to state our general expressions  for the B\'ezout coefficients, generalizing Theorem \ref{Fd-bigd-2} in the Introduction.
\begin{theorem}\label{Fd-alld}
Let $f,g\in \K[x]$  be monic polynomials of degrees $m$ and $n$, with multisets of roots $A$ and $B$,  and $\overline A$ and $\overline B$ are subsets of the sets of distinct roots of $f$ and $g$,  respectively, with $\overline m=|\overline A|$ and $\overline n=|\overline B|$, and set $m':=m-\overline m$ and $ n':=n-\overline n$.
For any $d$ such that $0\leq d < \min\{m,n\}$, we have
\begin{align*}&F_d(f,g)(x)=(-1)^{m-d +(n-d-1)(m-n')} \sum_{\substack{R_1\sqcup R_2\sqcup R_3=\{1,\dots, m' + n' -d \}\\
R_1\subset \{m+n-2d,\dots,m'+n'- d\},\\|R_1|=r_1, r_1\le d+2-(\overline m+\overline n)\\|R_2|=r_2,\, n'-d-1\le r_2\le n-d-1\\ |R_3|=r_3,\,m'-d\le r_3\le m-d}}(-1)^{\overline \sigma_{R}}\ \cdot \\[1mm]
& \qquad \cdot \sum_{\substack{A'\subset \overline A\\|A'|=r_3+\min\{n',d-m'\}}}
\sum_{\substack{B'\subset \overline B\\|B'|=r_2-(n' -d-1)}}
   \frac{\mathcal{R}(B\backslash \overline B, \overline A\backslash A')\mathcal{R}(\overline B\backslash B',A\backslash A')\mathcal{R}(x,\overline B\backslash B')}{ \mathcal{R}(A', \overline A\backslash A')\mathcal{R}(  B', \overline B\backslash B')}\,\cdot \\
 & \hskip2cm \cdot S_{d+1}^{(\widetilde R_1)}( A'\cup B')  S_{m+n-d}^{(R_2)}(A\cup (\overline{B}\backslash B')\cup \{x\}) S_{m+n-d}^{(R_3)}( (\overline{A}\backslash A')\cup B),\\
 \quad \\
 &G_d(f,g)(x)=(-1)^{(m-d-1)(d-m')} \sum_{\substack{R_1\sqcup R_2\sqcup R_3=\{1,\dots, m' + n' -d \}\\
R_1\subset \{m+n-2d,\dots,m'+n'- d\},\\|R_1|=r_1, r_1\le d+2-(\overline m+\overline n)\\|R_2|=r_2,\, m'-d-1\le r_2\le m-d-1\\ |R_3|=r_3,\,n'-d\le r_3\le n-d}}(-1)^{\widetilde \sigma_{R}}\ \cdot \\[1mm]
& \qquad \cdot \sum_{\substack{A'\subset \overline A\\|A'|=r_2-(m' -d-1)}}
\sum_{\substack{B'\subset \overline B\\|B'|=r_3+\min\{m',d-n'\}}}    \frac{\mathcal{R}(A\backslash \overline A, \overline B\backslash B')\mathcal{R}(\overline A\backslash A',B\backslash B')\mathcal{R}(x,\overline A\backslash A')}{ \mathcal{R}(A', \overline A\backslash A')\mathcal{R}(  B', \overline B\backslash B')}\,\cdot \\
 & \hskip2cm \cdot S_{d+1}^{(\widetilde R_1)}( A'\cup B')  S_{m+n-d}^{(R_2)}((\overline{A}\backslash A')\cup B\cup \{x\}) S_{m+n-d}^{(R_3)}(A\cup (\overline{B}\backslash B')),\end{align*}
 where
 \begin{align*}(-1)^{\overline \sigma_R}&:=(-1)^{  r_1(m-d+r_2+r_3) + r_2 \overline n  + r_3(m' +n' -d-1) +r_2r_3}\sg({R}),\\
 (-1)^{\widetilde \sigma_R}&:=(-1)^{  r_1(n-d+r_2+r_3) + r_2 \overline m  + r_3(m' +n' -d-1) +r_2r_3}\sg({R})
 \end{align*}
  and $$\widetilde R_1:=\{i-(m+n-2d-2):i\in R_1\}.$$
\end{theorem}

Again, we illustrate this result with a  toy example, leaving the proof to Section~\ref{s53} in the Appendix. We do not treat here  the case $f=(x-a_1)(x-a_2)^2$ and $g=(x-b)^3$ as we did for $\Sres_2(f,g)$ in the introduction, because in this case $F_2(f,g)=-1$ and $G_2(f,g)=1$ which do not have a lot of interest.
%

\begin{example}\label{example7}
Take  $f=(x-a)^3$ and $g=(x-b)^2$.    For $d=1$, we easily compute that   $$F_1(f,g)(x)= 1 \ \mbox{ and } \ G_1(f,g)(x)=-x+3a-b.$$ We now compute the expressions at the right-hand side in Theorem~\ref{Fd-alld}.\\
 For the first expression, corresponding to $F_1(f,g)(x)$, we have $R_1=R_2=\emptyset, \, R_3=\{1,2\}$. This gives $A'=\overline A$ and $B'=\overline B$. We then get that the sum equals $1$ and the initial sign equals $1$. That is, the right-hand side equals $1$, as expected.\\
 For the second expression, corresponding to $G_1(f,g)(x)$, the initial sign equals $-1$. We have    $R_1=\emptyset$, and  $ R_2=\{1\}$, $R_3=\{2\}$ or $ R_2=\{2\}$, $R_3=\{1\}$.  In the first case, $(-1)^{\tilde \sigma_R}=-1 $ while in the second case $(-1)^{\tilde \sigma_R}=1 $. In both cases $A'=\overline A$ and $B'=\overline B$, so what is left to compute in each sum is $S^{(R_2)}_4 (B\cup \{x\})$ and $S^{(R_3)}_4 (A)$.
 When $R_2=\{1\}$,  $S^{(R_2)}_4 (B\cup \{x\})=1$ and $S^{(R_3)}_4 (A)=-3a$, while when $R_2=\{2\}$, $$S^{(R_2)}_4 (B\cup \{x\})=\frac{\det\left(\begin{array}{ccc}b^3 & 3b^2 & x^3\\
 b & 1& x\\1& 0 &1\end{array}\right)}{\det\left(\begin{array}{ccc}b^2&2b&x^2\\b&1&x\\1&0&1\end{array}\right)}=x+2b,$$ and $S^{(R_3)}_4 (A)=1$.\\
 We finally get that the expression in the right-hand side equals
 $ 3a -x-2b$ which coincides with what is expected.
 \end{example}

\section{Comparisons with previous results}\label{s4}

In this paper we have succeeded in defining an expression in roots with multiplicities  $\SylM_{d}(A,B)(x)$ (Definitions \ref{def-bigd} and \ref{SylM}) which extends the classical Sylvester sum $\Syl_{d}(A,B)(x).$  However, in \cite{Syl} Sylvester also introduced the following {\em double sum}: for $0\leq p\leq m,\,0\leq q\leq n$,
{\small \begin{align*}
 \Syl_{p,q}(A,B)(x)&:=\sum_{\substack{A^{\prime
}\subset A,\,B'\subset
B\\|A^{\prime}|=p,\,|B'|=q}}\mathcal{R}(A^{\prime},B')\,
\mathcal{R}(A\backslash A^{\prime},B\backslash
B')\,\frac{\mathcal{R}(x,A^{\prime
})\,\mathcal{R}(x,B')}{\mathcal{R}(A^{\prime},A\backslash A^{\prime
})\,\mathcal{R}(B',B\backslash B')},
\end{align*}}
and showed that if we set $d:=p+q$; for $d\le \min\{m,n\}$ when $m\ne n,$ or $d< m=n$,
\begin{align*}\Syl_{p,q}(A,B)(x)&= (-1)^{p(m-d)}\binom{d}{p} \Sres_{d}(f,g)(x).
 \end{align*}
It would be interesting to produce expressions $\SylM_{p,q}(A,B)(x)$ for general multisets $A$ and $B$, which specialize to the above  double sums in the case of $A$ and $B$ being sets. Some extensions have been described in \cite[Section 4.2]{Val17} for the case $p$ and $q$ ``large enough'', but still more work has to be done in this direction.

\medskip
Recently, several explicit formulas ``in roots'' for univariate subresultants with multiplicities have been presented. We describe some of them and show that in all the cases, our $\SylM_{d}(A,B)(x)$ essentially produces new formulas.

\subsection{$\overline{m}=\overline{n}=1$}
In \cite{BDKSV17}, the authors, in a joint work with Alin Bostan, developed a formula for the subresultants in the extremal case when both polynomials $f$ and $g$ have only one (multiple) root each, that is when $f=(x-a)^m$ and $g=(x-b)^n$. More precisely,

\begin{theorem} \label{old} (\cite[Th.1.1, Th.1.2]{BDKSV17}) Let $m, n, d\in \N$ with $0\le d < \min\{m,n\}$, and $a, b\in \mathbb{K}$. Then
{  \begin{align*}
\Sres_d&((x-a)^m,(x-b)^n)(x)\\ &={(-1)}^{\binom{d}{2}}(a-b)^{(m-d)(n-d)} \sum_{j=0}^d q_j(m,n,d)(x-a)^j(x-b)^{d-j},
\end{align*}}
where the coefficients $q_0(m,n,d), \ldots, q_d(m,n,d)$ satisfy
\begin{align*}
q_0(m,n,d)&=(-1)^{\binom{d}{2}} \displaystyle{\prod_{i=1}^{d}}\dfrac{(i-1)!\,(m+n-d-i-1)!}{(m-i-1)!(n-i)!},\\
 q_j(m,n,d)&=
\frac{\binom{d}{j}\binom{n-d+j-1}{j}}{\binom{m-1}{j}} \, q_0(m,n,d)  \ \mbox{ for } \  1\le j\le d.\end{align*}
\end{theorem}

Comparing the expression above with the formula given in Theorem \ref{smalld} applied to the case $\overline{m}=\overline{n}=1,$ we get the following result:

\begin{proposition}\label{prop} Let $m, n, d\in \N$ with $1\le d < \min\{m,n\}$, and $a, b\in \mathbb{K}$. Let $A=\{\underbrace{a,\dots,a}_{m}\}  $ and $B=\{\underbrace{b,\dots,b}_{n}\} $. Then

{\small\begin{align*}\Sres_d&((x-a)^m,(x-b)^n)(x)\\ & =(a-b)^{n-1}(x-b)\mathcal{S}_{0}(d)+
(a-b)^{m-1}(x-a)\mathcal{S}_1(d)+
(x-a)(x-b)\mathcal{S}_2(d),\end{align*}}
where
\begin{itemize}
{\small\item $\mathcal{S}_0(d):=\displaystyle\sum_{\substack{R_2\cup R_3=\{1,\dots,m+n-2d-1\}\\
|R_2|=m-d-1,|R_3|=n-d}}(-1)^{\sigma_R} S_{m+n-d}^{(R_2)}(\{a\}\cup B)S_{m+n-d}^{(R_3)}(A),$
\item $\mathcal{S}_1(d):=\displaystyle\sum_{\substack{R_2\cup R_3=\{1,\dots,m+n-2d-1\}\\
|R_2|=m-d,|R_3|=n-d-1}}(-1)^{\sigma_R} S_{m+n-d}^{(R_2)}(B)S_{m+n-d}^{(R_3)}(A\cup\{b\}),$
\item  $\mathcal{S}_2(d)=0$ for $d=1$, and for $d>1$,
\begin{align*}
\mathcal{S}_2(d):=\displaystyle\sum_{i=1}^{d-1}&\,\,\,\displaystyle\sum_{\substack{R_2\cup R_3=\{1,\dots,m+n-2d-1\}\cup\{m+n-2d-1+i\}\\
|R_2|=m-d,|R_3|=n-d}}\\
&(-1)^{\sigma_R}S_{d+1}^{(\{1,\dots,d-1\}\backslash\{i\})}(\{a,b,x\}) S_{m+n-d}^{(R_2)}(B)S_{m+n-d}^{(R_3)}(A).\end{align*}}
\end{itemize}
\end{proposition}

As an example, for $m=3$, $n=2$, $\overline{m}=\overline{n}=1$ and $d=1$, we get from Theorem \ref{smalld} that
\begin{align*}
\Sres_1((x-a)^3,(x-b)^2)(x)&=(a-b)(x-b)(3a +\frac{-a^2-ab+2b^2}{a-b})\\ &\qquad \qquad \qquad +(a-b)^2(x-a)\\
&=(a-b)^2(2(x-b)+(x-a)),\end{align*}
which is consistent with the values $q_0(3,2,1)=2$ and $q_1(3,2,1)=1$ obtained in Theorem~\ref{old}.

Note however that the subresultant in Theorem \ref{old} is expressed as a linear combination of the family $(x-a)^j(x-b)^{d-j}$ while in Proposition \ref{prop} we get a combination of different powers $(x-a)^i(x-b)^j$ for $i+j\leq d$. This shows that these expressions are different except in the case where $d=1$ which appears below.

\medskip
\subsection{$d=m-1<n$}
In \cite{DKS2013}, the first three authors of this paper developed a general formula for the cases $d=m-1<n.$ Indeed, Proposition 2.6 in  \cite{DKS2013} states that for
$$f=(x-a_1)^{j_1}\cdots (x-a_{\overline m})^{j_{\overline m}},$$
with $j_1+\dots +j_{\overline m}=m$,
$\Sres_{m-1}(f,g)$ is the unique Hermite interpolant $h$ of degree $\le m-1$ satisfying the $m$ conditions
$$h^{(k_i)}(a_i)=g^{(k_i)}(a_i), \ 0\le k_i< j_i,  \ 1\le i\le \overline m .$$

We observe that the  formula described in Theorem~\ref{alld2}, expressed in terms of roots of $f$ {\em and} $g$, gives an alternative --though completely different and not at all obvious-- description of the Hermite interpolant $h$.  For instance, thanks to symmetric interpolation, we can check how  the polynomial $\SylM_{m-1}(A,B)(x)$ described in Definition~\ref{def-bigd}  satisfies (at least) the conditions $\SylM_{m-1}(A,B)(a_i)=(-1)^{m-1}g(a_i)$ for all $a_i\in \overline A$ when $n-\overline n \le \overline m-1$: from its definition,

{\small   {\begin{align*} &\SylM_{m-1}(A,B)(a_i)=\\
  &
   (-1)^{m-\overline m}
\sum_{\substack{B'\subset \overline B\\|B'|=m-\overline m }} \frac{\mathcal{R}(A\backslash \overline A, \overline B\backslash B')\mathcal{R}(a_i,  B\backslash B')\prod_{j\ne i}(a_i-a_j)\mathcal{R}(a_i,B')}{ \prod_{j\ne i}(a_j-a_i)\mathcal{R}(  B', \overline B\backslash B')
}\\ & = (-1)^{m-1} \mathcal{R}(a_i,  B')\sum_{\substack{B'\subset \overline B\\|B'|=m-\overline m }} \frac{\mathcal{R}(A\backslash \overline A, \overline B\backslash B')}{\mathcal{R}(  B', \overline B\backslash B')
}\\ & = (-1)^{m-1} g(a_i),\end{align*}}}
since by symmetric interpolation (Prop.~\ref{Lagr}), one has for $X=(x_1,\dots,x_{m-\overline m})$,
$$1= \sum_{\substack{B'\subset \overline B\\|B'|=m-\overline m }} \frac{\mathcal{R}(X, \overline B\backslash B')}{\mathcal{R}(  B', \overline B\backslash B')}= \sum_{\substack{B'\subset \overline B\\|B'|=m-\overline m }} \frac{\mathcal{R}(A\backslash \overline A, \overline B\backslash B')}{\mathcal{R}(  B', \overline B\backslash B')}.$$

\medskip
\subsection{$d=1$}
In the same paper \cite{DKS2013}, an explicit formula for $\Sres_1(f,g)(x)$ in terms of their multiple roots is given. To be more precise,  \cite[Thm.~2.7.]{DKS2013} states:
\begin{align*}
\Sres_1(f,g)(x)&=  (-1)^m g(a_1)\Big(\prod_{{2\le k\le m-1}}\frac{g(a_k)}{a_1-a_k}\Big)(x-a_1)\cdot \\ & \cdot \Big( \sum_{2\le k\le m-1}\frac{1}{a_1-a_k} + \frac{2}{a_1-b_1}+\sum_{2\le k\le n-1} \frac{1}{a_1-b_k} + 1\Big) \\ &+(-1)^{m-1}\sum_{2\le i\le m-1}\frac{g(a_1)^2}{(a_i-a_1)^2}\Big(\prod_{\substack{2\le k\le m-1\\
k\ne i}}\frac{g(a_k)}{a_i-a_k}\Big)(x-a_i),\end{align*}
which is an expression in the roots of $f$ and their values in $g$. On the other hand, Theorem~\ref{alld2} for $d=1$ and $m':=m- \overline m>0$, $n':=n- \overline n>0$, gives
{
\begin{align*}\Sres_1(f,g)(x)&=(-1)^{m-1}\Big(\mathcal{R}(\overline{A},B\backslash\overline{B})
\sum_{b\in\overline{B}}\frac{\mathcal{R}(A,\overline{B}\backslash b)}
{\mathcal{R}(b,\overline{B}\backslash b)}(x-b)\,\mathcal{T}_1 \\ & \quad
+\mathcal{R}(A\backslash\overline{A},\overline{B})
\sum_{a\in\overline{A}}\frac{\mathcal{R}(\overline{A}\backslash a,B)}
{\mathcal{R}(a,\overline{A}\backslash a)}(x-a)\,\mathcal{T}_2\Big),\end{align*}}
where
\begin{itemize}
{\small\item $\mathcal{T}_1:=\displaystyle\sum_{\substack{R_2\cup R_3=\{1,\dots,m'+n'-1\}\\
|R_2|=m'-1,|R_3|=n'}}(-1)^{\sigma_R} S_{m+n-1}^{(R_2)}(\overline{A}\cup B)
S_{m+n-1}^{(R_3)}(A\cup(\overline{B}\backslash b)),$
\item $\mathcal{T}_2:=\displaystyle\sum_{\substack{R_2\cup R_3=\{1,\dots,m'+n'-1\}\\
|R_2|=m',|R_3|=n'-1}}(-1)^{\sigma_R} S_{m+n-1}^{(R_2)}((\overline{A}\backslash a)\cup B)S_{m+n-1}^{(R_3)}(A\cup\overline{B}),$}
\end{itemize}
which are expressions in {\em both} the roots of $f$ and $g$. Note also that the first formulation is given as a linear combination of $(x-a),\,a\in A$ and constants, while the second one is a linear combination of $(x-a),\,a\in A$ {\em and} $(x-b),\,b\in B$. So their presentation is not the same.

\bigskip
\bibliographystyle{alpha}

\begin{thebibliography}{BDKSV2017}


\bibitem[ApJo2006]{ApJo06}
F. Ap\'ery, J.-P. Jouanolou.
\newblock{\em R\'esultant et sous-r\'esultants:
le cas d'une variable avec exercices corrig\'es.\/}
Hermann, Paris 2006. 477 p. (based on Cours DESS 1995-1996).

\bibitem[BDKSV2017]{BDKSV17}
A. Bostan, C. D'Andrea, T. Krick, A. Szanto, M. Valdettaro.
\newblock{\em Subresultants in multiple roots: An extremal case.\/}
\newblock{Linear Algebra Appl. 529 (2017), 185-198.}

\bibitem[Bor1860] {Bor60} C.W.~Borchardt.
\newblock {\"Uber eine Interpolationsformel f\"ur eine Art Symmetrischer
Functionen und \"uber Deren Anwendung}.
\newblock {\em Math. Abh. der Akademie der Wissenschaften zu Berlin}, pages
1--20, 1860.


\bibitem[Cha1990] {Cha90}M.~Chardin. \newblock Th\`ese. \newblock {\em Universit\'e Pierre
et Marie Curie (Paris VI)}, 1990.

\bibitem[ChLo1996]{CL96}{W. Y. Chen, J.D. Louck.}
     \newblock{\em Interpolation for symmetric functions.\/},
  \newblock{Adv. Math.  {117} (1996), no. 1,  147--156.}

\bibitem[DHKS2007]{DHKS07}
C. D'Andrea, H.  Hong, T. Krick, A. Szanto.
\newblock{\em  An elementary proof of Sylvester's double sums for subresultants.\/}
\newblock  J. Symbolic Comput.  42  (2007),  no. 3, 290--297.

\bibitem[DKS2013]{DKS2013}
C. D'Andrea, T. Krick, A. Szanto.
\newblock{\em Subresultants in Multiple Roots.}
\newblock Linear Algebra Appl. 438 (2013), no.5,  1969-1989.


\bibitem[DTGV2004] {DTGV04}G.M.~Diaz-Toca, L.~Gonzalez-Vega.
\newblock{\em  Various new expressions for subresultants and their applications.}
\newblock {\em Appl. Algebra Eng. Commun. Comput.}, 15(3--4):233--266, 2004.

\bibitem[Hon1999] {Hon99}H.~Hong. \newblock{\em Subresultants in roots.}
\newblock Technical report, Department of Mathematics. North Carolina State
University, 1999.

\bibitem[Kal1984]{Kal1984}
D. Kalman.
\newblock{\em The generalized Vandermonde matrix.\/}
\newblock Math. Mag.  57  (1984),  no. 1, 15--21.



\bibitem[KS2001]{KS01}
T. Krick, A. Szanto.
\newblock{\em Sylvester's double sums: an inductive proof of the general case.\/}
\newblock  J. Symbolic Comput.   47 (2012) 942--953.


\bibitem[KSV2017]{KSV16}
T. Krick, A. Szanto, M. Valdettaro.
\newblock{\em Symmetric interpolation, Exchange Lemma and  Sylvester sums.\/}
\newblock  Communications in Algebra  {\bf 45}, Issue 8 (2017) 3231-3250. DOI: 10.1080/00927872.2016.1236121



\bibitem[LaPr2001] {LaPr01}A.~Lascoux and P.~Pragacz.
\newblock Double sylvester sums for euclidean division, multi-Schur functions.
\newblock {\em Journal of Symbolic Computation}, (35):689--710, 2003.

\bibitem[Map2016]{maple}
Maple 2016.
\newblock{\em Maplesoft, a division of Waterloo Maple Inc., Waterloo, Ontario.\/}

\bibitem[RoSz2011]{RoSz11}
M.-F.~Roy, A.  Szpirglas.
\newblock{Sylvester double sums and subresultants.}
\newblock {\em J. Symbolic Comput.} 46:385--395.

\bibitem[Syl1839]{sylv39}
J. J. Sylvester.
\newblock {\em On rational derivation from equations of coexistence, that is to say, a new and extended theory of elimination.\/}
\newblock Philos. Mag. 15 (1839), 428--435.
\newblock Also appears in the Collected Mathematical Papers
          of James Joseph Sylvester, Vol.~{\bf 1},
Chelsea Publishing Co. (1973), 40–-46.


\bibitem[Syl1840]{sylv40}
J. J. Sylvester.
\newblock {\em A method of determining by mere inspection the derivatives from two equations of any degree.\/}
\newblock Philos. Mag. 16 (1840), 132--135.
\newblock Also appears in the Collected Mathematical Papers
          of James Joseph Sylvester, Vol.~{\bf 1},
 Chelsea Publishing Co. (1973), 54–-57.



\bibitem[Syl1840b]{sylv40b}
J. J. Sylvester.
\newblock {\em Note on elimination.\/}
\newblock Philos. Mag. 17 (1840), no. 11, 379--380.
\newblock Also appears in the Collected Mathematical Papers
          of James Joseph Sylvester, Vol.~{\bf 1},
 Chelsea Publishing Co. (1973), p.~58.


%


\bibitem[Syl1853]{Syl}
J.J. Sylvester.
\newblock {\em On a theory of syzygetic relations of two rational integral
functions, comprising an application to the theory of Sturm's
function and that of the greatest algebraical common measure.\/}
\newblock Philosophical Transactions of the Royal Society of London, Part III
(1853), 407--548.
\newblock Appears also in Collected Mathematical Papers
          of James Joseph Sylvester, Vol. {\bf 1},
 Chelsea Publishing Co. (1973)  429--586.

 \bibitem[Val17]{Val17}
 M.~Valdettaro.
 \newblock{\em F\'ormulas en ra\'ices para las subresultantes.\/}
 \newblock Ph.D. Thesis, University of Buenos Aires, 2017.
{\tt http://cms.dm.uba.ar/academico/carreras/doctorado/tesis-Valdettaro.pdf}

\end{thebibliography}
\def\cprime{$'$} \def\cprime{$'$} \def\cprime{$'$}

\bigskip
\appendix
\section{}\label{s5}
\subsection{Proof of Proposition~\ref{ApeJou}}\label{s51}

The proof follows from a suitable extension of the next Exchange Lemma that appears in  \cite[Lem.3.1 \& Cor.3.2]{KSV16}.

\begin{lemma}\label{exch}  Set $d\ge 0$. Let $A,\, B\subset \K$ be finite sets with $|A|, |B|\ge d$,  and  $X$ a set of variables with  $|X|\le |A|-d$.   Then

\[\sum_{{A'\subset  A,
|A'|=d}}\mathcal{R}(A\backslash A^{\prime},B)\frac{\mathcal{R}(X,A^{\prime})}{\mathcal{R}(A\backslash A^{\prime},A^{\prime})}= \sum_{B^{\prime}\subseteq B,
|B^{\prime}|=d}\mathcal{R}(A,B\backslash B^{\prime})\frac{\mathcal{R}(X,B^{\prime})}{\mathcal{R}(B', B\backslash B^{\prime})} .\]
\end{lemma}

Lemma~\ref{exch}   turns out to be  a consequence of the symmetric interpolation developed in \cite{CL96} (see also \cite{KSV16}) that we state here as we will need it for the proof of Lemma~\ref{exchsymII}.

\begin{proposition} \label{Lagr}  Let $E\subset \K$ be a finite set of size $|E|=e$. Set $0\le d<e,$  and let
$X$ be a set of variables with  $|X|=e-d$. Then, \[ {\mathcal B}: =\big\{\mathcal{R}(X,E')\,;\, E'\subseteq E, |E'|=d \big\}\]
is a basis of the $\K$-vector space $S_{(e-d,d)}$ of symmetric polynomials $h$ in $X=\{x_1,\dots, x_{e-d}\}$ over $\K$  such that $\deg_{x_i}(h)\le d$ for  all $1\le i\le e-d$. \\Moreover, any polynomial $h(X)\in S_{(e-d,d)}$ can be uniquely written  as
\[h(X) = \sum_{{ E'\subseteq E,
|E'|=d}} h(E\backslash E') \frac{\mathcal{R}(X,E')}{{\mathcal{R}(E\backslash
E^{\prime }, E^{\prime})}}\]
where $h(E\backslash E^{\prime}):=h(e_1,\dots,e_{e-d})$ for $E\backslash
E^{\prime }=\{e_1,\dots,e_{e-d}\}$.
\end{proposition}

Our next extension of Lemma~\ref{exch} relaxes slightly the condition on the size of $X$. Item (2) is presented for sake of completeness, we do not use it in the sequel.
\begin{lemma} \label{exchsymII}
Set $d\ge 0$. Let $A,B\subset \K$ be  finite sets  with  $|A|\ge d,$ and
$X$ be a set of variables with $|X|\le |A|+ |B|-2d$. Then
\begin{enumerate}
\item If $|B|\ge d$, then
\begin{align*}\sum_{\substack{A_1\sqcup A_2 = A\\
|A_1|=d, |A_2|=|A|-d}}&\frac{\mathcal{R}(A_2,B)\mathcal{R}(X,A_1)}{\mathcal{R}(A_1,A_2)}=\\ & = (-1)^{d(|A|-d)}
\sum_{\substack{B_1\sqcup B_2=B\\
|B_1|=d,\,|B_2|=|B|-d}}\frac{\mathcal{R}(A,B_2)\mathcal{R}(X,B_1)}{\mathcal{R}(B_1,B_2)}.\end{align*}
\item If $|B|<d$, then
\[\sum_{\substack{A_1\sqcup A_2 = A\\
|A_1|=d,\,|A_2|=|A|-d}}\frac{\mathcal{R}(A_2,B)\mathcal{R}(X,A_1)}{\mathcal{R}(A_1,A_2)}=0.\]
\end{enumerate}
\end{lemma}

\begin{proof}  (1) When  $|B|\ge d$, if  $|X|\le |A|-d$ holds, we are in the conditions of Lemma \ref{exch} and the statement holds by simply correcting the sign.

Now assume $|B|\ge d$ and  $r:=|X|>|A|-d$. Write $X=Y\cup Z$, with $Y=\{x_1,\cdots,x_{|A|-d}\}$ and $Z=\{x_{|A|-d+1},\cdots,x_r\}$. We define
\[h(Y,Z)=\sum_{\substack{A_1\sqcup A_2 = A\\
|A_1|=d, |A_2|=|A|-d}}\frac{\mathcal{R}(A_2,B)\mathcal{R}(Y,A_1)\mathcal{R}(Z,A_1)}{\mathcal{R}(A_1,A_2)},\]
and
\[g(Y,Z)=\sum_{\substack{B_1\sqcup B_2=B\\
|B_1|=d,\,|B_2|=|B|-d}}\frac{\mathcal{R}(A,B_2)\mathcal{R}(Y,B_1)\mathcal{R}(Z,B_1)}{\mathcal{R}(B_1,B_2)},\]
and show that $h=(-1)^{d(|A|-d)}g$. For this purpose, we consider $g,h\in \K(Z)[Y]$, i.e. with coefficients in the field $\K(Z).$  Both polynomials are symmetric in $Y$ and have multidegree in $Y$ bounded by $d$. So $h,g\in S_{n-d,d}(\K(Z))$. Using Proposition~\ref{Lagr}, it is enough to verify that $h(A_2,Z)=(-1)^{d(|A|-d)}g(A_2,Z)$, for all $A_2\subseteq A$ with $|A_2|=|A|-d$. Clearly, for a given $A_2$,  $h(A_2,Z)=(-1)^ {d(|A|-d)}\mathcal{R}(A_2,B)\mathcal{R}(Z,A_1)$ where $A_1:=A\backslash A_2$. Let us compute $g(A_2,Z)$:

\begin{align*}g(A_2,Z)&=\sum_{\substack{B_1\sqcup B_2=B\\
|B_1|=d,\,|B_2|=|B|-d}}\frac{\mathcal{R}(A, B_2)\mathcal{R}(A_2,B_1)\mathcal{R}(Z,B_1)}{\mathcal{R}(B_1, B_2)}\\ &=
\mathcal{R}(A_2,B)\sum_{\substack{B_1\sqcup B_2=B\\
|B_1|=d,\,|B_2|=|B|-d}}\frac{\mathcal{R}(A_1,B_2)\mathcal{R}(Z,B_1)}{\mathcal{R}(B_1,B_2)}.\end{align*}
Thus it suffices to show that
\begin{equation}\label{id:z}\sum_{\substack{B_1\sqcup B_2=B\\
|B_1|=d,\,|B_2|=|B|-d}}\frac{{\mathcal{R}(A_1,B_2)}\mathcal{R}(Z,B_1)}{\mathcal{R}(B_1,B_2)}=\mathcal{R}(Z,A_1).\end{equation}
But this   holds again by Lemma \ref{exch} for $B$ instead of $A$, $A_1$ instead of $B$ and $Z$ instead of $X$, since $|Z|=|X|-(|A|-d)\leq |B|-d$ by hypothesis  (in this case the only subset of $A_1$ of size $d$ is $A_1$ itself).

 (2) When $|B|<d$, we enlarge $B$ by adding variables $Y$ so that $|B\cup Y|=d$, say  $Y=\{y_1,\cdots,y_s\}$, with $s=d-|B|$. So we get, by applying the previous item, that
{\small{\begin{align*}&\sum_{\substack{A_1\sqcup A_2 = A\\
|A_1|=d, |A_2|=|A|-d}}\frac{\mathcal{R}(A_2,B)\mathcal{R}(X,A_1)}{\mathcal{R}(A_1,A_2)}=\\ & =  (-1)^{(|A|-d)|Y|}\,\,\coeff_{y_1^{|A|-d}\cdots y_s^{|A|-d}}\left(\sum_{\substack{A_1\sqcup A_2 = A\\
|A_1|=d, |A_2|=|A|-d}}\frac{\mathcal{R}(A_2,B\cup Y)\mathcal{R}(X,A_1)}{\mathcal{R}(A_1,A_2)}\right)\\ & =(-1)^{(|A|-d)|Y|}\,\,\coeff_{y_1^{|A|-d}\cdots y_s^{|A|-d}}\mathcal{R}(X,B\cup Y)\ = \
0 ,\end{align*}}}
\noindent since in this case the hypothesis $|X|\le |A|+|B|-2d$ together with $|B|<d$ implies that $|X|<|A|-d$, and therefore there is no coefficient in $y_i$  of degree $|A|-d$.
\end{proof}

 \noindent
\begin{proof}[Proof of Proposition~\ref{ApeJou}.] Set $e:=|E|$, $m:=|A|$ and $n:=|B|$. The right-hand side of the equality  we want to show can be rewritten as
{\small{\begin{align}&\sum_{\substack{E_1\sqcup E'=E\\
|E_1|=d,\,|E'|=e-d}}\,\,\sum_{\substack{E_2\sqcup E_3=E'\\
|E_2|=m-d,\,|E_3|=e-m}}\frac{\mathcal{R}(A, E_3)\mathcal{R}(E_2,B)\mathcal{R}(X,E_1)}
{\mathcal{R}(E_1,E')\mathcal{R}(E_2,E_3)}
\nonumber\\&=\sum_{\substack{E_1\sqcup E'=E\\
|E_1|=d,\,|E'|=e-d}}\frac{\mathcal{R}(X,E_1)}{\mathcal{R}(E_1,E')}
\sum_{\substack{E_2\sqcup E_3=E'\\
|E_2|=m-d,\,|E_3|=e-m}}\frac{\mathcal{R}(A,E_3)\mathcal{R}(E_2,B)}{\mathcal{R}(E_2,E_3)}
\nonumber\\&=(-1)^{m(e-m)+n(m-d)}\sum_{\substack{E_1\sqcup E'=E\\
|E_1|=d,\,|E'|=e-d}}\frac{\mathcal{R}(X,E_1)}{\mathcal{R}(E_1,E')}
\sum_{\substack{E_2\sqcup E_3=E'\\
|E_2|=m-d,\,|E_3|=e-m}}\frac{\mathcal{R}(E_3,A)\mathcal{R}(B,E_2)}{\mathcal{R}(E_2,E_3)}
\nonumber \\ &= (-1)^{d(e-m)+n(m-d)}  \sum_{\substack{E_1\sqcup E'=E\\
|E_1|=d,\,|E'|=e-d}}\frac{\mathcal{R}(X,E_1)}{\mathcal{R}(E_1,E')}
\sum_{\substack{A_2\sqcup A_1=A\\
|A_2|=m-d,\,|A_1|=d}}\frac{\mathcal{R}(E',A')\mathcal{R}(B,A_2)}
{\mathcal{R}(A_2,A_1)} \label{una} \\
 & =(-1)^{d(e-m)} \sum_{\substack{A_2\sqcup A_1=A\\
|A_2|=m-d,\,|A_1|=d}} \frac{\mathcal{R}(A_2,B)}{\mathcal{R}(A_2,A_1)}\sum_{\substack{E_1\sqcup E'=E\\
|E_1|=d,\,|E'|=e-d}}\frac{\mathcal{R}(E',A_1)\mathcal{R}(X,E_1)}
{\mathcal{R}(E_1,E')}
\nonumber
\\
& =(-1)^{d(m-d)}  \sum_{\substack{A_2\sqcup A_1=A\\
|A_2|=m-d,\,|A_1|=d}} \frac{\mathcal{R}(A_2,B)}{\mathcal{R}(A_2,A_1)}\,
\mathcal{R}(X,A_1) \label{dos}\\
&= \sum_{\substack{A_1\sqcup A_2=A\\
|A_1|=d,\,|A_2|=m-d}} \frac{\mathcal{R}(A_2,B)\mathcal{R}(X,A_1)}{\mathcal{R}(A_1,A_2)}\nonumber,
\end{align}}}
where ~\eqref{una} is Lemma~\ref{exchsymII}(1) applied to  $E'$ instead of $A$,  $A$ instead of $B$ and $B$ instead of $X$   since $|B|\le |E'|+|A|-2(m-d)$, i.e. $n\le e-m+d$ by hypothesis, and ~\eqref{dos} is the same lemma applied to $E$ instead of $A$, $A_1$ instead of $B$ and $X$ since $|X|\le |E|+|A_1|-2d$, i.e. $|X|\le e-d$ by hypothesis (note that in this case,  the only subset of $A_1$ of size $d$ is $A_1$ itself and therefore  the second sum in Lemma~\ref{exchsymII} simply equals $\mathcal{R}(X,A_1)$).
\end{proof}

\subsection{Proof of Theorem~\ref{smalld}}\label{s52}

As in the proof of Theorem~\ref{bigd}, we can assume that $A\cap B=\emptyset$. The idea of the proof is to add an auxiliary set of variables $T=\{t_1,\cdots,t_r\}$ with $r=m'+n' -d$ so that   $E:=\overline A\cup \overline B\cup T$ has size $|E|=m+n-d$, which allows us to apply Proposition~\ref{ApeJou} to $E$ and $X=\{x\}$, and then to compare coefficients in the obtained expression.\\
Applying Proposition~\ref{ApeJou} we get
\[\sum_{\substack{A_1\sqcup A_2 = A\\
|A_1|=d,\\|A_2|=m-d}}\frac{\mathcal{R}(A_2,B)\mathcal{R}(x,A_1)}{\mathcal{R}(A_1,A_2)}=\sum_{\substack{E_1\sqcup E_2\sqcup E_3 = \overline A \cup \overline B\cup T\\
 |E_1|=d, |E_2|=m-d\\
 |E_3|=n-d}}\frac{\mathcal{R}(A,E_3)\mathcal{R}(E_2,B)\mathcal{R}(x,E_1)}{\mathcal{R}(E_1,E_2)
 \mathcal{R}(E_1,E_3) \mathcal{R}(E_2,E_3)}.\]

Again, $\mathcal{R}(A,E_3)=0$ when $E_3\cap A\neq \emptyset$ and $\mathcal{R}(E_2,B)=0 $ when $E_2\cap B\neq\emptyset$. Therefore $E_3\subset \overline B\cup T$ and $E_2\subset \overline A \cup T$. Let us write  $E_2= (\overline A\backslash A')\cup T_2$ with $A'\subset \overline A$ and $T_2\subset T$,  $E_3= (\overline B\backslash B')\cup T_3$ with $B'\subset \overline B$ and $T_3\subset T$ with $T_2\cap T_3=\emptyset$. Then $E_1= (A'\cup B')\cup T_1$ where $T_1=T\backslash (T_2\cup T_3)$, and we can rewrite the sum as we did  in Theorem~\ref{bigd}:

 {\small{\begin{align*}\label{inicio} &\sum_{\substack{T_1\sqcup T_2\sqcup T_3=T\\
|T_1|=r_1, 0\le r_1\le d\\|T_2|=r_2, 0\le r_2\le m-d\\ |T_3|=r_3,0\le r_3\le n-d}}\,\sum_{\substack{A'\subset \overline A\\|A'|=  r_2+d-m'\\0\le |A'|\le \overline m}}
\sum_{\substack{B'\subset \overline B\\|B'|=r_3+d-n'\\0\le |B'|\le \overline n}}\\ &  \frac{\mathcal{R}(A,(\overline B\backslash B')\cup T_3)\mathcal{R}((\overline A\backslash A')\cup T_2,B)\mathcal{R}(x,(A'\cup B')\cup T_1)}{ \mathcal{R}((A'\cup B')\cup T_1, (\overline A\backslash A')\cup T_2)\mathcal{R}( (A'\cup B')\cup T_1, (\overline B\backslash B')\cup T_3)\mathcal{R}((\overline A\backslash A')\cup T_2, (\overline B\backslash B')\cup T_3)
}\\
&=\sum_{\substack{T_1\sqcup T_2\sqcup T_3=T\\
|T_1|=r_1, 0\le r_1\le d\\|T_2|=r_2, \max\{0,m'-d\}\le r_2\le m-d\\ |T_3|=r_3,\max\{0,n'-d\}\le r_3\le n-d}}\,\sum_{\substack{A'\subset \overline A\\|A'|= r_2-(m'-d)}}
\sum_{\substack{B'\subset \overline B\\|B'|=r_3-(n' -d)}}\\ &
 \frac{\mathcal{R}(A,(\overline B\backslash B')\cup T_3)\mathcal{R}((\overline A\backslash A')\cup T_2,B)\mathcal{R}(x,(A'\cup B')\cup T_1)}{ \mathcal{R}((A'\cup B')\cup T_1, (\overline A\backslash A')\cup T_2)\mathcal{R}( (A'\cup B')\cup T_1, (\overline B\backslash B')\cup T_3)\mathcal{R}((\overline A\backslash A')\cup T_2, (\overline B\backslash B')\cup T_3)
}
\end{align*}}}
Here, for each choice of $T_1,T_2,T_3$ and $A',B'$, the numerator equals
$$\mathcal{R}(A,\overline B\backslash B')\mathcal{R}(A, T_3)\mathcal{R}(\overline A\backslash A',B)\mathcal{R}( T_2,B)\mathcal{R}(x,A')\mathcal{R}(x, B')\mathcal{R}(x,T_1),$$
while the denominator can be rewritten as
\begin{align*} &
 \mathcal{R}(A'\cup B', \overline A\backslash A')\mathcal{R}(A'\cup B', T_2)\mathcal{R}( T_1, \overline A\backslash A')\mathcal{R}( T_1,  T_2) \cdot \\ & \quad \cdot\mathcal{R}( A'\cup B', \overline B\backslash B')\mathcal{R}( A'\cup B', T_3)\mathcal{R}( T_1, \overline B\backslash B')\mathcal{R}(  T_1, T_3)\cdot \\ &\quad \cdot \mathcal{R}(\overline A\backslash A', \overline B\backslash B') \mathcal{R}(\overline A\backslash A', T_3) \mathcal{R}(T_2, \overline B\backslash B')\mathcal{R}( T_2, T_3).
\end{align*}
Therefore, the part of the quotient which is  free of $T_\ell$'s equals, as in Theorem~\ref{bigd},
{\small{\begin{align*}
&
= (-1)^{\sigma_1}\sum_{\substack{A'\subset \overline A\\|A'|=r_2-(m'-d) }}
\sum_{\substack{B'\subset \overline B\\|B'|=r_3-(n' -d) }} \frac{\mathcal{R}(A\backslash \overline A, \overline B\backslash B')\mathcal{R}(\overline A\backslash A',B\backslash B')\mathcal{R}(x,A')\mathcal{R}(x,B')}{ \mathcal{R}(A', \overline A\backslash A')\mathcal{R}(  B', \overline B\backslash B')
}\end{align*}}}
with $\sigma_1:=|B'|\, |\overline A\backslash A'|$.\\
We deal now with the part of the quotient that involves some $T_\ell$. Multiplying and dividing by $\mathcal{R}(T_1,A'\cup B')\mathcal{R}(T_2,\overline A\backslash A')\mathcal{R}(T_3,\overline B\backslash B')$, we get that this quotient equals
$$(-1)^{\sigma_2}\frac{\mathcal{R}(T_3,A\cup (\overline B\backslash B'))\mathcal{R}(T_2, (\overline A\backslash A')\cup B)\mathcal{R}(T_1,A'\cup B'\cup x)}
 {\mathcal{R}(T,\overline{A}\cup\overline{B})\mathcal{R}(T_1,T_2)\mathcal{R}(T_1,T_3)\mathcal{R}(T_2,T_3)},$$
 where $\sigma_2:=|T_3|\,|\overline A\backslash A'| + (|T_2|+|T_3|)|A'\cup B'|+|T_3|\,|A| +|T_1|.$\\
Next we multiply and divide by the product of Vandermonde determinants $\det(V(T_1))\det(V(T_2))\det(V(T_3))$, where we consider in each $T_\ell$ the elements $t_i$ with the indices $i$ in increasing order, and get
{\small{\begin{align*}&\frac{\mathcal{R}( T_3,A\cup (\overline B\backslash B'))\mathcal{R}(T_2,(\overline A\backslash A')\cup B)\mathcal{R}(T_1,A'\cup B'\cup x)\det(V(T_1))\det(V(T_2))\det(V(T_3))}
 {\mathcal{R}(T,\overline{A}\cup\overline{B})\mathcal{R}(T_1,T_2)\mathcal{R}(T_1,T_3)
 \mathcal{R}(T_2,T_3)\det(V(T_1))\det(V(T_2))\det(V(T_3))}=\\
&\quad = \sg(T_1,T_2,T_3)\cdot \\
& \qquad \cdot
\frac{\mathcal{R}( T_3,A\cup (\overline B\backslash B'))\mathcal{R}(T_2,(\overline A\backslash A')\cup B)\mathcal{R}(T_1,A'\cup B'\cup x)\det(V(T_1))\det(V(T_2))\det(V(T_3))}
 {\mathcal{R}(T,\overline{A}\cup\overline{B})\det(V(T))},\end{align*}}}
 where $\sg(T_1,T_2,T_3):=(-1)^\sigma$ where $\sigma$ is the parity of the number of transpositions needed to bring the ordered set  $T_1\sqcup T_2 \sqcup T_3 $ to $\{t_1,\dots,t_r\}$.

 Since the denominator is independent of the choices of $T_\ell$, going back to the first expression, we have

{\small{\begin{align*}\mathcal{R}(T,\overline{A}\cup\overline{B})&\det(V(T))\sum_{\substack{A_1\sqcup A_2 = A\\
|A_1|=d,\,|A_2|=m-d}}\frac{\mathcal{R}(A_2,B)\mathcal{R}(x,A_1)}{\mathcal{R}(A_1,A_2)}=\\
&=
\sum_{\substack{T_1\sqcup T_2\sqcup T_3=T\\
|T_1|=r_1, 0\le r_1\le d\\|T_2|=r_2, \max\{0,m'-d\}\le r_2\le m-d\\ |T_3|=r_3,\max\{0,n'-d\}\le r_3\le n-d}}(-1)^{\sigma'}\sg(T_1,T_2,T_3)\,\cdot \\ &\cdot \, \sum_{\substack{A'\subset \overline A\\|A'|=r_2-(m'-d)) }}
\sum_{\substack{B'\subset \overline B\\|B'|=r_3-(n' -d) }} \frac{\mathcal{R}(A\backslash \overline A, \overline B\backslash B')\mathcal{R}(\overline A\backslash A',B\backslash B')\mathcal{R}(x,A')\mathcal{R}(x,B')}{ \mathcal{R}(A', \overline A\backslash A')\mathcal{R}(  B', \overline B\backslash B')
}\cdot \\
 &\cdot \,\mathcal{R}( T_3,A\cup (\overline B\backslash B'))\mathcal{R}(T_2,(\overline A\backslash A')\cup B)\mathcal{R}(T_1,A'\cup B'\cup x)\cdot \\
 & \cdot \, \det(V(T_1))\det(V(T_2))\det(V(T_3)),
\end{align*}}}
where
{\small{\begin{align*}\sigma':&= \sigma_1+\sigma_2\\
&= (|B'|+|T_3|)|\overline A\backslash A'|+   (|T_2|+|T_3|)|A'\cup B'|+|T_3|\,|A| +|T_1|\\
&\equiv (n' -d)(m-d) + r_1 + r_2(m' -d+1) + r_3(n'-\overline m  +1) \pmod 2\\
&\equiv m'(m-d) + r_1(m-d-1) + r_2(\overline m -1) + r_3(m'+n' -d-1) \pmod 2. \end{align*}}}
(The last row is written in  a way that it coincides with the exponent in Theorem~\ref{bigd}, when $r<0$ is interpreted as $r_1=r_2=r_3=0$.)

To recover the sum we are looking for, we take a specific coefficient in $(t_1,\dots,t_r)$ in both sides. Note that the leading coefficient of $\mathcal{R}(T,\overline{A}\cup\overline{B})\det(V(T))$ w.r.t. the lexicographic term order $t_1>\dots > t_r$ equals
$$ \coeff_{t_1^{m+n-d-1}t_2^{m+n-d-2}\cdots t_r^{m+n-d-r}}\left(\mathcal{R}(T, \overline{A}\cup\overline{B})\det(V(T))\right)=1.$$

We now look for  this coefficient on the right hand side of the expression above. We do it by keeping track of the variables $t_i$ that belong to each $T_\ell$.
 We go first after the variables in $T_2$, and then in $T_3$, since they behave similarly. Observe that   \[\mathcal{R}(T_2,(\overline{A}\backslash A')\cup B)\det(V(T_2))=\frac{\det(V(T_2\cup (\overline{A}\backslash A')\cup B))}{\det(V((\overline{A}\backslash A')\cup B))},\]
and
\[\mathcal{R}(T_3,A\cup (\overline{B}\backslash B'))\det(V(T_3))=\frac{\det(V(T_3\cup A\cup (\overline{B}\backslash B')))}{\det(V(A\cup(\overline{B}\backslash B')))},\]
where the matrices in the numerator of the right-hand sides are both  of size $(m+n-d)\times (m+n-d)$.
The coefficient of the monomial $\prod_{t_i\in T_2} t_i^{m+n-d-i}$  corresponds  to the submatrix of $V_{m+n-d}((\overline{A}\backslash A')\cup B)$ where the rows indexed by  $R_2:=\{i:\,t_i\in T_2\}$ have been  erased. Then
{\small{\begin{align*}\coeff_{{{\prod_{t_i\in T_2} t_i^{m+n-d-i}}}}\left(\frac{\det(V(T_2\cup (\overline{A}\backslash A')\cup B))}{\det(V((\overline{A}\backslash A')\cup B))}\right)&=  \sg_{m+n-d}(R_2) S_{m+n-d}^{(R_2)}((\overline{A}\backslash A')\cup B)\\
&=  \sg_{m'+n'-d}(R_2) S_{m+n-d}^{(R_2)}((\overline{A}\backslash A')\cup B)\end{align*}}}
since $R_2\subset \{1,\dots,m'+n'-d\}$.
Analogously
{\small{\[\coeff_{{{\prod_{t_i\in T_3} t_i^{m+n-d-i}}}}\left(\frac{\det(V(T_3\cup A\cup (\overline{B}\backslash B')))}{\det(V(A\cup(\overline{B}\backslash B')))}\right)= \sg_{m'+n'-d}(R_3)S_{m+n-d}^{(R_3)}(A\cup (\overline{B}\backslash B')),\]}}
where $R_3:=\{i: \,t_i\in T_3\}$.

Now we deal with variables in $T_1$. Note that
\[\mathcal{R}(T_1,  A'\cup B'\cup x) \det(V(T_1))=\frac{\det(V(T_1\cup A'\cup B'\cup x))}{\det(V(A'\cup B'\cup x))}.\]
Here the matrix in the numerator is a Vandermonde matrix of size  $(d+1)\times(d+1)$ and the maximal exponent of $t_i$ for $t_i\in T_1$ that can appear equals $t_i^d$.  Set $R_1:=\{i:\, t_i\in T_1\}$. Therefore, for all $i\in R_1$ one needs $m+n-d-i\subset \{0,1,\dots, d\}$, i.e. $m+n-2d\le i\le m+n-d$. Since $i\le r=m'+n' - d,$ we need $m+n-2d\le m'+n' -d$ and $R_1\subset \{m+n-2d,\dots, m'+n' -d\}$.
\\In particular, when  $m+n-2d> m'+n' -d $, i.e. when $d<\overline m+\overline n$ there is no choice of $R_1$. In that case, we conclude

{\small{\begin{align*}&\sum_{\substack{A_1\sqcup A_2 = A\\
|A_1|=d,\\|A_2|=m-d}}\frac{\mathcal{R}(A_2,B)\mathcal{R}(x,A_1)}{\mathcal{R}(A_1,A_2)}\,=\, (-1)^{m'(m-d)} \sum_{\substack{R_2\sqcup R_3=\{1,\dots, m'+n' -d \}\\
|R_2|=r_2, \max\{0,m'-d\}\le r_2\le m-d\\ |R_3|=r_3,\max\{0,n'-d\}\le r_3\le n-d}}(-1)^\sigma\,\sg(R_2,R_3)\cdot \\ & \qquad \cdot \sum_{\substack{A'\subset \overline A\\|A'|= r_2-(m'-d)}}
\sum_{\substack{B'\subset \overline B\\|B'|=r_3-(n' -d)}}  \frac{\mathcal{R}(A\backslash \overline A, \overline B\backslash B')\mathcal{R}(\overline A\backslash A',B\backslash B')\mathcal{R}(x,A')\mathcal{R}(x,B')}{ \mathcal{R}(A', \overline A\backslash A')\mathcal{R}(  B', \overline B\backslash B')}\,\cdot \\
 & \hskip5cm \cdot S_{m+n-d}^{(R_2)}((\overline{A}\backslash A')\cup B) S_{m+n-d}^{(R_3)}(A\cup (\overline{B}\backslash B')),\end{align*}}}
 where $$\sigma:=   r_2(\overline m -1) + r_3(m'+n' -d-1) +r_2r_3,$$
 since it is easy to check that $\sg_{m'+n'-d}(R_2)\sg_{m'+n'-d}(R_3)=(-1)^{r_2r_3} $ as $R_2$ and $R_3$ are complementary sets in $\{1,\dots,m'+n'-d\}$ (or see Lemma~\ref{signs} below).

Now,  when $d\ge \overline m+\overline n$ and $R_1=\{i:\, t_i\in T_1\}\subset \{m+n-2d,\dots, m'+n' -d\}$  we have
\[\coeff_{\prod_{t_i\in T_1} t_i^{m+n-d-i}}\left(\frac{\det(V(T_1 \cup A'\cup B'\cup x))}{\det(V( A'\cup B'\cup x))}\right)=\sg_{d+1}(\widetilde R_1)S_{d+1}^{(\widetilde R_1)} (A'\cup B'\cup x),\]
where $\widetilde R_1:=\{i- (m+n-2d-1):\, i\in R_1\} \subset \{1,\dots,d+1-(\overline m+\overline n)\}$.
We prove in Lemma~\ref{signs} below that $$\sg_{d+1}(\widetilde R_1)\sg_{m'+n'-d}(R_2)\sg_{m'+n'-d}(R_3)=(-1)^{r_1(r_2+r_3+m+n-1)+r_2r_3}.$$
 Therefore we get
{\small{\begin{align*}&\sum_{\substack{A_1\sqcup A_2 = A\\
|A_1|=d,\\|A_2|=m-d}}\frac{\mathcal{R}(A_2,B)\mathcal{R}(x,A_1)}{\mathcal{R}(A_1,A_2)}\, = \, (-1)^{m'(m-d)} \,\cdot \\ & \qquad \cdot  \sum_{\substack{R_1\sqcup R_2\sqcup  R_3=\{1,\dots,m'+n' -d\}\\
 R_1\subset\{m+n-2d , \dots,m'+n' -d\}, \\| R_1|=r_1, 0\le r_1\le d-(\overline m+\overline n )+1\\ |R_2|=r_2, \max\{0,m'-d\}\le r_2\le m-d\\ |R_3|=r_3,\max\{0,n'-d\}\le r_3\le n-d }} (-1)^\sigma\,\sg(R_1,R_2,R_3)\,\\[1mm]& \qquad \cdot \sum_{\substack{A'\subset \overline A\\|A'|=  r_2-(m' -d)}}
\sum_{\substack{B'\subset \overline B\\|B'|= r_3-(n' -d)}}  \frac{\mathcal{R}(A\backslash \overline A, \overline B\backslash B')\mathcal{R}(\overline A\backslash A',B\backslash B')\mathcal{R}(x,A')\mathcal{R}(x,B')}{ \mathcal{R}(A', \overline A\backslash A')\mathcal{R}(  B', \overline B\backslash B')}\,\cdot \\
 & \qquad \qquad \qquad  \cdot \,   S_{d+1}^{(\widetilde R_1)}(A'\cup B'\cup x)\,S_{m+n-d}^{(R_2)}((\overline{A}\backslash A')\cup B)\, S_{m+n-d}^{(R_3)}(A\cup (\overline{B}\backslash B')),\end{align*}}}
where $$\sigma:=  r_1(n-d+r_2+r_3) + r_2(\overline m -1) + r_3(m'+n' -d-1)+r_2r_3.$$

\begin{lemma}\label{signs} Let $R_1\sqcup R_2\sqcup  R_3$ be a partition of  $\{1,\dots,r\}$ with  $|R_i|=r_i$ for $1\le i\le 3$,  and
$0\le s\le r$ be such that $\widetilde R_1=\{i-s:\, i\in R_1\}\subset \{1,\dots, r-s\}$. Then
$$\sg_{r-s}(\widetilde R_1)\,\sg_{r}(R_2)\,\sg_{r}(R_3)=(-1)^{r_1(r_2+r_3+s)+r_2r_3}.$$
\end{lemma}
\begin{proof} We set $R_1=\{i_1,\dots,i_{r_1}\}$, $R_2=\{j_1,\dots,j_{r_2}\}$ and $R_3=\{k_1,\dots,k_{r_3}\}$. Then
{\small \begin{align*}\sg_{r-s}(\widetilde R_1)\,\sg_{r}(R_2)\,\sg_{r}(R_3)&= \sum_{1\le \ell \le r_1}(i_\ell-s-\ell)+\sum_{1\le \ell \le r_2}(j_\ell-\ell)+\sum_{1\le \ell \le r_3}(k_\ell-\ell)\\
&= \frac{r(r+1)}{2} - r_1s - \frac{r_1(r_1+1)}{2}- \frac{r_2(r_2+1)}{2}- \frac{r_3(r_3+1)}{2} \\
&= \frac{r^2-r_1^2-r_2^2-r_3^2}{2} - r_1s\\
& =\frac{r^2-(r_1+r_2+r_3)^2 + 2r_1r_2+2r_1r_3+2r_2r_3 }{2} - r_1s\\
 & \equiv r_1r_2+r_1r_3+r_2r_3 +r_1s\pmod2.
\end{align*}}
\end{proof}

\subsection{Proof of Theorem~\ref{Fd-alld}}\label{s53}

First note that the statement of Theorem~\ref{Fd-alld} implies the expressions  in Theorem \ref{Fd-bigd-2} for the case when $m'+n'\leq d$ since in this case $r_1=r_2=r_3=0$. Here we prove the statement for $G_d(f,g)$,  the statement for $F_d(f,g)$ in  follows from  the identity
$$ F_d(f,g)=(-1)^{(m-d)(n-d)} G_d(g,f).$$
By Lemma~\ref{lemmaGd} and Theorem~\ref{alld2} we have
\begin{equation}\label{Gd3}G_d(f,g)(x)=(-1)^{(d+1)(m-d-1)} \coeff_{y^{d+1}} \big(\SylM_{d+1}(A,B\cup\{x\})(y)\big).\end{equation}
Now we  use Definition \ref{SylM} for $A$ and $B\cup \{x\}$   and  choose  $\overline A\subset A$ and $\overline B\subset B\subset B\cup\{x\}$ as subsets of the sets of distinct roots in $f$ and $\tilde g$ respectively, and get that
{\small \begin{align*}
&\SylM_{d+1}(A,B\cup\{x\})(y)= (-1)^{m'(m-d) }  \cdot \\& \cdot \sum_{\substack{R_1\sqcup R_2\sqcup R_3=\{1,\dots, m' + n' -d \}\\
R_1\subset \{m+n-2d-1,\dots,m'+n'- d\},\\|R_1|=r_1, r_1\le d+2-(\overline m+\overline n)\\|R_2|=r_2,\, m'-d-1\le r_2\le m-d-1\\ |R_3|=r_3,\,n'-d\le r_3\le n-d}}(-1)^{\sigma_{R}}\frac{\mathcal{R}(A\backslash \overline A, \overline B\cup\{x\}\backslash B')\mathcal{R}(\overline A\backslash A',B\cup\{x\}\backslash B')\mathcal{R}(y,A')\mathcal{R}(y,B')}{ \mathcal{R}(A', \overline A\backslash A')\mathcal{R}(  B', \overline B\backslash B')}\,\cdot \\
 & \hskip2cm  \cdot \,  S_{d+2}^{(\widetilde R_1)}(A'\cup B'\cup \{y\})S_{m+n-d}^{(R_2)}((\overline{A}\backslash A')\cup B\cup\{x\})  S_{m+n-d}^{(R_3)}(A\cup (\overline{B}\backslash B')).\end{align*}}
 We have to consider $S_{d+2}^{(\widetilde R_1)}(A'\cup B'\cup\{y\})$ for $R_1\subset \{m+n-2d-1,\dots, m'+n'-d\}$ and $\widetilde R_1=\{i-(m+n-2d-2): i\in R_1\}$.
Since  $\deg_y\big( \mathcal{R}(y,A'\cup B'))=|A'|+|B'| $, we first observe that $$\deg_y\big(S_{d+2}^{(\widetilde R_1)}(A'\cup  B'\cup \{y\})\big)\le d+1-|A'|-|B'|$$ and moreover, if $1\in \widetilde R_1$, i.e, if $m+n-2d-1\in R_1$, then
$$\coeff_{y^{d+1-|A'|-|B'|}}\big(S_{d+2}^{(\widetilde R_1)}(A'\cup  B'\cup \{y\})\big) = 0,$$
 while if
$1\notin \widetilde R_1$, i.e, if $m+n-2d-1\notin R_1$, then
$$\coeff_{y^{d+1-|A'|-|B'|}}\big(S_{d+2}^{(\widetilde R_1)}(A'\cup  B'\cup \{y\})\big) = (-1)^{|A'|+|B'|} S_{d+1}^{(\widetilde R_1)}(A'\cup  B').$$
Therefore the only subsets $R_1$ that produce non-zero terms satisfy  $R_1\subset \{m+n-2d,\dots,m'+n'-d\}$ and for these $R_1$,
$$\coeff_{y^{d+1}}\big( \mathcal{R}(y,A'\cup B') S_{d+2}^{(\widetilde R_1)}(A'\cup  B'\cup \{y\}) \big)= (-1)^{|A'|+|B'|} S_{d+1}^{(\widetilde R_1)}(A'\cup  B').$$
Hence,
\begin{align*}&\coeff_{y^{d+1}}\big(\SylM_{d+1}(A,B\cup\{x\})(y)\big)=(-1)^{m'(m-d-1)}\,\cdot  \\ & \qquad \qquad \cdot \, \sum_{\substack{R_1\sqcup R_2\sqcup R_3=\{1,\dots, m' + n' -d \}\\
R_1\subset \{m+n-2d,\dots,m'+n'- d\},\\|R_1|=r_1, r_1\le d+2-(\overline m+\overline n)\\|R_2|=r_2,\, m'-d-1\le r_2\le m-d-1\\ |R_3|=r_3,\,n'-d\le r_3\le n-d}}(-1)^{\sigma_{R}}\ \cdot \\[1mm]
& \qquad \cdot \sum_{\substack{A'\subset \overline A\\|A'|=r_2-(m' -d-1)}}
\sum_{\substack{B'\subset \overline B\\|B'|=r_3-(n' -d)}}    \frac{\mathcal{R}(A\backslash \overline A, \overline B\backslash B')\mathcal{R}(\overline A\backslash A',B\backslash B')\mathcal{R}(\overline A\backslash A',x)}{ \mathcal{R}(A', \overline A\backslash A')\mathcal{R}(  B', \overline B\backslash B')}\,\cdot \\
 & \hskip2cm \cdot S_{d+1}^{(\widetilde R_1)}( A'\cup B') S_{m+n-d}^{(R_2)}((\overline{A}\backslash A')\cup B\cup \{x\}) S_{m+n-d}^{(R_3)}(A\cup (\overline{B}\backslash B')),\end{align*}
 where for the partition $R_1\sqcup R_2\sqcup R_3=\{1,\dots, m'+n'-d\}$ and $R=(R_1,R_2,R_3)$
   \begin{equation*}(-1)^{\sigma_{ R}}=(-1)^{ r_1(n-d+r_2+r_3)+r_2(\overline m -1) + r_3(m' +n' -d-1) +r_2r_3}\sg({R}).\end{equation*}
We conclude the proof by applying again Identity~\eqref{Gd3} and by permuting $x$ and $\overline A\backslash A'$ in $\mathcal{R}(\overline A\backslash A',x)$:
 \begin{align*}&G_d(f,g)(x)=(-1)^{(d-m')(m-d-1)} \sum_{\substack{R_1\sqcup R_2\sqcup R_3=\{1,\dots, m' + n' -d \}\\
R_1\subset \{m+n-2d,\dots,m'+n'- d\},\\|R_1|=r_1, r_1\le d+2-(\overline m+\overline n)\\|R_2|=r_2,\, m'-d-1\le r_2\le m-d-1\\ |R_3|=r_3,\,n'-d\le r_3\le n-d}}(-1)^{\widetilde \sigma_{R}}\ \cdot \\[1mm]
& \qquad \cdot\sum_{\substack{A'\subset \overline A\\|A'|=r_2-(m' -d-1)}}
\sum_{\substack{B'\subset \overline B\\|B'|=r_3-(n' -d)}}    \frac{\mathcal{R}(A\backslash \overline A, \overline B\backslash B')\mathcal{R}(\overline A\backslash A',B\backslash B')\mathcal{R}(x,\overline A\backslash A')}{ \mathcal{R}(A', \overline A\backslash A')\mathcal{R}(  B', \overline B\backslash B')}\,\cdot \\
 & \hskip2cm \cdot S_{d+1}^{(\widetilde R_1)}( A'\cup B')  S_{m+n-d}^{(R_2)}((\overline{A}\backslash A')\cup B\cup \{x\}) S_{m+n-d}^{(R_3)}(A\cup (\overline{B}\backslash B')),\end{align*}
 where $(-1)^{\widetilde \sigma_R}:=(-1)^{  r_1(n-d+r_2+r_3) + r_2 \overline m  + r_3(m' +n' -d-1) +r_2r_3}\sg({R}) $.

\end{document}